\documentclass[10pt]{article}
\usepackage{times}
\usepackage{graphicx}
\usepackage[T1]{fontenc}
\usepackage{amsmath,amssymb,amsfonts,amsthm,amscd}
\usepackage{eucal}
\usepackage[english]{babel}
\usepackage[utf8]{inputenc}
\usepackage{enumerate}
\usepackage{url}
\usepackage{mathpazo}
\usepackage{mathrsfs}
\usepackage{stmaryrd}
\usepackage{mathtools}
\usepackage{xcolor}
\usepackage{marvosym}
\usepackage{subcaption}
\usepackage[margin=1.0in]{geometry}
\usepackage{import}
\usepackage{xifthen}
\usepackage{pdfpages}
\usepackage{transparent}
\usepackage[small]{titlesec}
\usepackage{enumitem}
\usepackage{esint}
\usepackage[colorlinks]{hyperref}
\usepackage{pgfplots}
\pgfplotsset{compat=1.18}

\theoremstyle{plain}
\numberwithin{equation}{section}

\newtheorem{theorem}{Theorem}
\newtheorem{lemma}[theorem]{Lemma}

\newtheorem{proposition}[theorem]{Proposition}
\newtheorem{definition}[theorem]{Definition}
\newtheorem{remark}[theorem]{Remark}

\newtheorem{example}[theorem]{Example}

\newcommand{\tf}[1]{\overset{\circ}{#1}}

\newcommand{\R}{\mathbb{R}}

\newcommand{\Imag}{\mathrm{Im}\,}

\newcommand{\g}[2]{\left\langle #1,#2\right\rangle}
\newcommand{\gBig}[2]{\Big\langle #1,#2\Big\rangle}
\newcommand{\gbig}[2]{\big\langle #1,#2\big\rangle}
\newcommand{\norm}[1]{\left\Vert #1\right\Vert}
\newcommand{\abs}[1]{\left| #1\right|}

\newcommand{\p}[1]{{\left( #1\right)}}
\newcommand{\pbig}[1]{{\big( #1\big)}}

\newcommand\SF{\text{\Gemini}}
\newcommand\MC{\mathcal{H}}

\newcommand\PS{\partial M}
\newcommand\Lie{\mathcal{L}}
\newcommand{\sq}[1]{\left[#1\right]}

\newcommand{\tb}[1]{\left\{#1\right\}}
\newcommand{\tr}{\mathrm{tr}\,}
\newcommand{\trgt }{\mathrm{tr}_{g^T}}

\DeclareMathOperator{\Ric}{Ric}
\DeclareMathOperator{\scal}{scal}
\DeclareMathOperator{\trace}{tr}

\colorlet{linkequation}{blue}

\newcommand*{\SavedEqref}{}
\let\SavedEqref\eqref
\renewcommand*{\eqref}[1]{%
  \begingroup
    \hypersetup{
      linkcolor=linkequation,
      linkbordercolor=linkequation,
    }%
    \SavedEqref{#1}%
  \endgroup
}

\begin{document}

\author{Rasmus J. Jouttij\"arvi\\ KTH}
\title{On the Stability of Einstein Manifolds with Boundary}
\maketitle

\textbf{Abstract:} 
We study the stability problem for Einstein manifolds with boundary with respect to the Einstein-Hilbert action. The geometric boundary conditions we are using arise naturally from studying the calculus of variations associated with Ricci flow, \cite{RJJ}. Upon the introduction of a boundary, the space of $TT_g$ tensors no longer arise naturally as the defining space for the stability condition. Thus we must settle for the larger subspace of tensors preserving the scalar curvature, the total volume and the Bianchi gauge condition, which we call the space of  $TV_g$ tensors. As a test-case, we shall discuss stability of the Riemannian Schwarzschild anti-deSitter family of metrics; a problem known as "a Black hole in a box".

\tableofcontents

\pagebreak

\section{Introduction}

Throughout, $M$ shall denote an $(n\geq 3)$-dimensional manifold with closed compact boundary $\partial M$. We shall assume that $M$ is equipped with a Riemannian metric $g$, while $\partial M$ is equipped with the metric $g^T$ induced by $g$. Recall that $g$ is called Einstein when its Ricci curvature satisfies
$$\Ric_g=\mu g\qquad \text{for some }\;\; \mu \in \mathbb{R}$$
Let $\gamma$ be a Riemannian metric on $\partial M$, and let $H:\partial M\to \R$ be an arbitrary function. These will be assumed smooth, unless otherwise stated. Let $\mathcal{M}$ be a section of the set of metrics on $M$, whose regularity we will specify when needed. The pair $(\gamma,H)$ is in itself unimportant, save that it allows us to define the space of metrics
$$\mathcal{M}_B:=\tb{g\in \mathcal{M}\; :\;g^T=e^{\Psi}\gamma,\; \MC_g=e^{-\frac{\Psi}{2}}H\; \;\text{on} \;\PS}.$$
Here, $\MC_g$ denotes the scalar mean curvature of $g$, and $\Psi$ is a function introduced to give a scale factor between $g^T$ and $\MC_g$.\\

In Proposition \ref{prop:split} we show that when $g\in \mathcal{M}_B$ is Einstein and $\frac{\scal_g}{n-1}\notin \sigma_N(\Delta)$, the Neumann spectrum of the Laplacian, the tangent space at $g$ decomposes into 
$$T_{g}\mathcal{M}_B=C^\infty_{N}(M)\cdot g\oplus \ker_g D_g\scal,$$
where $C^\infty_{N}(M)$ is the space of smooth Neumann functions, and $\ker_g$ denotes kernel elements of zero total trace. The decomposition allows us to conclude the following slice theorem, Theorem \ref{MB_cor}:
\begin{theorem} Let $\mathcal{M}_B$ be of H\"older regularity $C^{k,\alpha}$, and let $\mathcal{C}_B\subset\mathcal{M}_B$ be the space of constant scalar curvature metrics. Suppose $g\in \mathcal{C}_B$ is Einstein and $\frac{\scal_g}{n-1}\notin \sigma_N(\Delta)$, then there is a local diffeomorphism 
$$C_{g,N}^{k,\alpha}(M)\times \mathcal{C}_B\longrightarrow \mathcal{M}_B,$$
in a neighbourhood of $g$. Specifically, any metric close to $g$ is related, via a Neumann function with zero mean, to a constant scalar curvature metric.  
\end{theorem}
By combining the above decomposition with another, we refine the tangent space, in Proposition \ref{bianchisplit}, to
\begin{equation}\label{intro:split}T_g\mathcal{M}_B=C^\infty_{N}(M)\cdot g\oplus TV_g\oplus \mathrm{Im}\; \delta^*|_{\Omega_0}.\end{equation}
where $TV_g$ is the Bianchi gauged subspace of $\ker_g D_g\scal$, and $\Imag\delta^*|_{\Omega_0}$ is the space of infinitesimal diffeomorphisms fixing the boundary. In this context, the Bianchi gauge refers to the kernel of the Bianchi operator, $\beta_g=\delta_g +\frac{1}{2}d\mathrm{tr}_g$, which also appears under a variety of other names in the literature (deDonder, deTurck, Harmonic, etc.). This decomposition is a natural analogue of the standard decomposition in the closed case;
$$T_g\mathcal{M}=C^\infty(M) \cdot g\oplus TT_g\oplus\Imag \delta^*,$$
when $M$ is closed (see Remark \ref{noboundsplitting}).\\

Let $g_t$ be a curve of metrics in $\mathcal{M}_B$ with $g_{t=0}=g$ Einstein. Then the second variation of the Einstein--Hilbert action $S_g$ depends only on the linear term $h:=g'_{t=0}\in T_g\mathcal{M}_B$;
$$S''_g(h)=-\frac{1}{2}\p{F_gh,h}_{L^2},$$
where
$$F_gh:=\Delta_E h-2\delta^*_g\beta_g h-D_g\scal (h)\cdot g.$$
As it determines the extremality properties of $g$ as a (semi-)critical point, it is called the Einstein--Hilbert stability operator. $F_g$ is non-trivial on the first two factors of \eqref{intro:split}, and in most cases it will have a definite sign on the first. However, as one may not always rule out negative eigenvalues for the operator on $C^\infty_{N}(M)\cdot g$, we shall distinguish between \emph{conformal stability} and \emph{mode stability} - that which is defined on $TV_g$. On this space the stability operator reduces to the usual elliptic Einstein operator $\Delta_E$. We prove in Proposition \ref{prop:fredholm} of Section \ref{sec:stab}, that the given boundary conditions leads to a well-posed elliptic eigenvalue problem. Specifically, $\Delta_E$ is a self-adjoint Fredholm operator on $TV_g$.\\

The fact that positive Einstein manifolds with boundary may be infinitesimally conformally unstable, is one of the major differences between this and the closed case. In that case $F_g$ is  positive definite on the conformal part of the tangent space for every Einstein metric $g$, apart from the round sphere, for which $F_g$ is positive semi-definite. However, we shall show that it is still possible to obtain an Obata-type uniqueness statement for Einstein manifolds with boundary (see \cite{Oba} for the original, and \cite{Aku} for a similar result).

\begin{theorem} Let $g_1,g_2\in \mathcal{M}_B$ be Einstein metrics in the same conformal class, $\sq{g_1}=\sq{g_2}$. If $g_1$ is not isometric to a round hemisphere, then $g_1=cg_2$ for some $c>0$. If $g_1$ is isometric to a hemisphere, then the same conclusion holds, up to isometry.
\end{theorem}

The remaining sections will be dedicated to some explicit examples. First, in Section \ref{Sec:confunstab}, we exhibit two examples of conformally unstable Einstein metrics, adaptable to arbitrary dimension $n\geq 3$. \\

In Section \ref{sec:SADS}, we prove mode stability of the family of Riemannian Schwarzschild anti-deSitter metrics in a spherical cavity. We obtain the following stability theorem, Theorem \ref{thm:SADS}:
\begin{theorem} Every member of the Schwarzschild anti-deSitter family of metrics in dimension $n\geq 4$ is stable in a spherical cavity of radius $R=\p{(n-1)m}^{1/(n-3)}$, with respect to spherically symmetric perturbations. The stability is strict when the Einstein constant is non-zero.
\end{theorem}

\noindent Furthermore, in the case of the four dimensional Schwarzschild metric, our choice of boundary conditions resolves the "discrepancy" between the radius at which the black hole is thermodynamically stable, $R=3m$, and the radius at which the isothermal eigenmode appears, $R\sim 2.89m$ \cite{All}. See also \cite{Pre} for commentary.
\begin{theorem}
The four dimensional Schwarzschild metric is unstable for $R>3m$.
\end{theorem}
\noindent A more precise formulation can be found in Theorem \ref{thm:S}.\\

The linearised boundary conditions on $T_g\mathcal{M}_B$ are related to what is known as \emph{Anderson boundary data}. In \cite{An1} and \cite{An2}), Michael Anderson argued for the use of the conformal class of the boundary and mean curvature as boundary conditions for problems related to Einstein metrics. His work was subsequently extended in \cite{AH}, and adapted for the Ricci flow in \cite{Gia}. Anderson's conditions also show up as a special case of the family of boundary conditions considered in \cite{LSW}.

\section{Preliminaries}

\subsection{Notation}

For $h\in S^2M$ a symmetric two-tensor, the notation $h^T\in S^2\partial M$ refers to the tangential part of $h$ at the boundary. A particular case is the induced metric on the boundary $g^T$.\\
  
We shall apply the notation of \cite{Bes} for the divergence operator $\delta_g=-\mathrm{div_g}$, as well as its formal adjoint $\delta^*_g$ and the geometers Laplacian $\Delta_g=\delta_gd$. For all three operators, we shall sometimes omit the subscript when there can be no question about the metric. The co-divergence, $\delta^*_g$, is the symmetrised covariant derivative and it satisfies the relation $\delta_g^*\alpha=\frac{1}{2}\mathcal{L}_{\alpha^\sharp}g$, for every one form $\alpha\in\Omega ^1(M)$. \\

Let $h\in S^2M$. We denote the trace-free part of $h$ as 
$$\overset{\circ}{h}=h-\frac{\tr h}{n}g.$$
In a few notable cases, we use $\overset{\circ}{\cdot }$ in conjunction with $\cdot^T$, by which we shall always mean
$$\overset{\circ}{h}\!\:^T=h^T-\frac{\tr_{g^T}h}{n-1}g^T.$$
A returning favourite is the Einstein tensor 
$$E_g:=\overset{\circ}{\Ric_g}=\Ric_g-\frac{\scal_g}{n} g,$$
the trace-less part of the Ricci tensor, and an integral part of the decomposition of the full curvature tensor:
$$R_g=\frac{1}{2n}\frac{\scal_g}{n-1}\; g\owedge g+\frac{1}{n-2}\; E_g\owedge g+W_g,$$
where $W_g$ is the (conformal) Weyl tensor and $\owedge$ is the Kulkarni--Nomizu product. Note the distinction between $E_g$ and the transverse tensor usually seen in Einstein's field equations $G_g=\Ric_g-\frac{1}{2}\scal_g\cdot g$.\\

Let $X$ be a vector field and $h$ a symmetric tensor, then we shall adhere to Obata's notation for the contraction
$$h\cdot X:=h(X,\cdot)\in \Omega^1(M).$$
We allow $X$ to be substituted by a one form $\alpha\in \Omega^1(M)$, by setting $h\cdot \alpha:=h\cdot \alpha^{\sharp}$.\\

Let $\nu$ be the outward pointing unit normal form along the boundary. We use the following convention for the second fundamental form and mean curvature of $g$
$$\SF_g:=\p{\delta_g^*\nu}^T\qquad \text{and}\qquad \MC_g:=\trgt \SF_g.$$
When calculating the first variation of $\SF_g$, it is beneficial to use the Leibniz property of the Lie derivative;
\begin{equation}\label{eq:SFvariation}\begin{split}
   \SF_g'(h)&= \frac{1}{2}\p{\Lie_{\nu^\sharp}g^T}'=\frac{1}{2}\Lie_{\nu^\sharp}h^T+\frac{1}{2}\Lie_{\p{\nu^\sharp}'}g^T\\
   &=\frac{1}{2}\Lie_{\nu^\sharp}h^T-\Lie_{(h\cdot\nu)^T}g^T-\frac{1}{4}h(\nu,\nu)\Lie_{\nu^\sharp}g^T\\ &=\frac{1}{2}\nabla_\nu h^T+\frac{1}{2}\SF_g\times h^T-\delta_{g^T}^*\p{h\cdot \nu}-\frac{1}{2}h\p{\nu,\nu}\SF_g,\end{split}
\end{equation}
where $h\cdot \nu=h(\nu^\sharp,\cdot)$, $\p{\SF\times h^T}_{\alpha\beta}=g^{\gamma\delta}\p{\SF_{\alpha \gamma}h_{\delta \beta}+h_{\alpha \gamma}\SF_{\delta\beta}}$, and we used $(\nu^\sharp)'=-\p{h\cdot \nu}^T-\frac{1}{2}h\p{\nu,\nu}\nu^\sharp$, the derivation of which may be found in \cite[Lemma 2.2]{RJJ}. Using the Leibniz property of the trace, we obtain a formula for the first variation of the mean curvature:
\begin{align}\label{MCvar}
\MC_g'(h)&=\tr_{g^T}\SF_g'(h)-\g{h^T}{\SF_g}=\frac{1}{2}\nu\p{\trgt {h}}+\delta_{g^T}\p{h\cdot \nu}-\frac{1}{2}h(\nu,\nu)\MC_g.
\end{align}

We let $\mathcal{M}$ denote the space of smooth Riemannian metrics on $M$. In a few cases, we shall need to consider the section of metrics with H\"older regularity $C^{k,\alpha}$, with $k\in \mathbb{N}_0$ and $\alpha\in (0,1)$, up to the boundary. When this is the case, we shall write $\mathcal{M}^{k,\alpha}$.\\

    Let $\Omega=\Omega^1(M)$ be the space of smooth one forms on $M$, then we define
    $$\Omega_0:=\tb{\omega\in \Omega\; :\; \omega|_{\partial M}\equiv 0}.$$
    The space of smooth Neumann functions, with vanishing mean, shall be denoted 
    $$C_{g,N}^\infty\p{M}:=\tb{f\in C^\infty(M)\;:\; \int_M f\;dV_g=0\;\;\&\;\; \nu\p{f}=0\;\; \text{on}\;\; \partial M}.$$
    We shall also employ the notation $C_{N}^\infty(M)$ and $C_g^\infty(M)$, when we drop the vanishing mean or Neumann condition, respectively. Analogously to the space of metrics, we shall use $\Omega^{k,\alpha}$ and $C^{k,\alpha}(M)$ to denote the spaces of H\"older regularity.\\
     
     By $\sigma_N(\Delta)$ we shall mean the discrete spectrum of the Laplacian acting on functions $\psi\in C^\infty_{g,N}(M)$. We stress that this implies $\sigma_N\p{\Delta}\subset (0,\infty)$, as the spectrum does does not include $0$.\\

Finally we recall the definition of the Einstein--Hilbert action:
\begin{definition} We define the Einstein--Hilbert action (or total scalar curvature) as the functional $S:\mathcal{M}\to \R$;
$$S(g):=\int_M\scal_g\; dV_g.$$    
\end{definition}

\subsection{Boundary Conditions}
We shall consider a coupled pair of boundary conditions, the first of which is a generalised Dirichlet condition;
\begin{equation}\label{BC1}h^T=\frac{1}{n-1}\trgt h\cdot g^T.\end{equation}
This should be interpreted as fixing the conformal class of the induced metric on the boundary. In itself, it will not be enough for a well-defined boundary value problem for the linearised Einstein equations. The solution is to amend the Dirichlet condition by a generalised Neumann condition, in the form of a variational equation for the mean curvature;
\begin{equation}\label{BC2geometric}
    2\MC_g'(h)+\frac{1}{n-1}\trgt h\;\MC_g=0.
\end{equation}
The two conditions are coupled in the sense that the degree of freedom introduced by allowing a arbitrary conformal factor $\varphi=\frac{1}{n-1}\trgt h$, is fixed by the second condition.\\

While \eqref{BC2geometric} is the most geometric representation of the Neumann boundary condition, it often arises from a boundary integral where the integrand is given by $\g{\delta h+d\tr h}\nu$. In fact, the difference between this and \eqref{BC2geometric} is precisely a boundary divergence term;
\begin{equation}\label{BC2nongeometric}
    \g{\delta_g h+d\mathrm{tr}_g{h}}{\nu}+\delta_{g^T}\p{h\cdot \nu}=2\MC_g'(h)+\g{h^T}{\SF_g}=2\MC_g'(h)+\varphi\MC_g.
\end{equation}
As we shall always assume the boundary to be closed, we have
$$\int_{\PS} \g{\delta_g h+d\mathrm{tr}_g{h}}{\nu}\; dA_g=\int_{\PS}2\MC'_g(h)+\varphi \MC_g\; dA_g.$$
The identity \eqref{BC2nongeometric} is proven as follows. Extending $\nu$ geodesically to a neighbourhood of $\partial M$, we can decompose the normal divergence as
\begin{align*}\g{\delta_g h}{\nu}&=\delta\p{h\cdot \nu}+\g{h}{\delta^*_g\nu}\\ &=\delta_{g^T}(h\cdot \nu)-h(\nu,\nu)\MC_g-\nu\p{h\p{\nu,\nu}}+\g{h^T}{\SF_g}.\end{align*}
Add to this the normal derivative of the trace, and we obtain
$$\g{\delta_g h}{\nu}+\nu(\mathrm{tr}_g h)=\delta_{g^T}\p{h\cdot \nu}-h(\nu,\nu)\MC+\nu\p{\trgt h}+\g{h^T}{\SF_g},$$
which, combined with \eqref{MCvar} and $\g{h^T}{\SF_g}=\varphi\MC_g$, gives the identity. A third and final version of the second boundary condition uses the Bianchi operator, $\beta_g h:=\delta_g\p{h-\frac{1}2\mathrm{tr}_g h\cdot g}$, 
\begin{equation*}
    2\MC_g'(h)+\varphi\MC_g=\g{2\beta_g h}{\nu}+\p{h\p{\nu,\nu}-\varphi}\MC_g+\nabla_\nu h(\nu,\nu).
\end{equation*}

Even though the boundary conditions are defined at a linearised level, they are in fact integrable in the sense that we can define a space of metrics whose tangent space at any point is the space of symmetric tensors satisfying the conditions. Let $\gamma$ be a background metric on $\PS$ and let $H:\PS\to \R$ be a chosen function. Define the space of boundary conformal metrics
    \begin{equation}\label{MB}\mathcal{M}_B:=\tb{g\in \mathcal{M}\; :\; g^T=e^{\psi}\gamma,\; \MC_g=e^{-\frac{\psi}{2}}H\;\; \text{at}\;\; \partial M}.\end{equation}
    It is safe to assume that the metrics in $\mathcal{M}_B$ are smooth, unless stated otherwise.\\ 
    
    A consequence of the choice of boundary conditions, is the fact that the metrics $g_1,g_2\in\mathcal{M}_B$ are conformally equivalent if and only if there exists a Neumann function $f\in C^\infty_N(M)$ such that $g_1=e^{f}g_2$. This has a couple of important implications: First, let $[g]\subset \mathcal{M}_B$ be the subspace of metrics conformal to $g$, then
    $$T_g[g]=C_N^\infty(M)\cdot g =\mathbb{R}\cdot g\oplus_{L^2}C^{\infty}_{g,N}(M)\cdot g. $$
    Secondly, having the elliptic Neumann boundary condition on conformal perturbations allow us to retain certain useful results from the theory of stability of closed Einstein manifolds. This is discussed in the following section, but is most evident in Theorem \ref{Escobar4.2}.

\subsection{Escobar's Eigenvalue Estimate}
A useful tool in the study of stability of positive Einstein metrics on a closed compact manifold is Obata's eigenvalue estimate \cite[Theorem 1]{Oba}:
\begin{equation}\label{Obata}\lambda_1(\Delta)\geq \frac{n\mu}{n-1},\end{equation}
where $\lambda_1(\Delta)> 0$ is the principal eigenvalue of the Laplacian and $\mu\geq 0$ is the Einstein constant, $\Ric_g=\mu g$. Equality in \eqref{Obata} is obtained if and only if $(M,g)$ is isometric to $\mathbb{S}^n$ with the usual round metric $g_{\mathbb{S}^n}$. \\

In the case of a manifold with boundary, \eqref{Obata} holds for the principal Neumann eigenvalue, but only if we assume the boundary is convex. In this case, equality holds if and only if $(M,g)$ is isometric to $\p{\mathbb{S}_+^n(r), g_{\mathbb{S}^n}}$ (where $r=\sqrt{\frac{n-1}{\mu}}$). We shall refer to this as Escobar's eigenvalue estimate (See \cite[Theorem 4.3]{Esc}). The argument is straightforward and mimics the original proof of \eqref{Obata}; Let $f\in C_{g,N}^\infty(M)$, then integration by parts yields
$$\norm{\Delta f}^2_{L^2}=\p{\nabla \Delta f,\nabla f}_{L^2}=\p{\Delta \nabla f,\nabla f}_{L^2}+\int_M\Ric_g(\nabla f,\nabla f)\; dV.$$
Performing another instance of integration by parts and a Ricci identity yields 
$$\p{\Delta \nabla f,\nabla f}_{L^2}=\Vert\nabla^2 f\Vert_{L^2}^2-\int_{\partial M}\nabla^2f\p{d f,\nu}\; dA=\Vert\nabla^2 f\Vert_{L^2}^2+\int_{\partial M}\SF\p{\nabla f,\nabla f}\; dA.$$
If $\Ric_g\geq \mu g$, as is the case for Einstein metrics, and $\partial M$ is convex, i.e. $\SF\geq 0$, we conclude that 
\begin{equation}\label{laphessineq}
\norm{\Delta f}_{L^2}^2-\Vert\nabla^2f\Vert_{L^2}^2\geq \mu \norm{\nabla f}_{L^2}^2.\end{equation}
If we let $f$ be the principal Neumann eigenfunction for the Laplacian, Cauchy--Schwarz and \eqref{laphessineq} yield
$$\frac{n-1}{n}\lambda_1^2\norm{f}_{L^2}^2=\frac{n-1}{n}\norm{\Delta f}_{L^2}^2\geq \norm{\Delta f}_{L^2}^2-\Vert\nabla^2f\Vert_{L^2}^2\geq \mu \norm{\nabla f}_{L^2}^2, $$
or
$$\lambda_1^2\geq \frac{n\mu}{n-1}\frac{\norm{\nabla f}_{L^2}^2}{\norm{f}_{L^2}^2}=\frac{n\mu}{n-1}\lambda_1.$$
Equality occurs if and only if 
$$\nabla^2f=-\frac{\mu f}{n-1} g.$$
As we shall see below in Theorem \ref{Escobar4.2}, the existence of such an $f$ implies that $(M,g)$ is isometric to a hemisphere with the round metric. \\

There are structurally simple examples of Einstein manifolds with non-convex boundary, for which \eqref{Obata} is violated. Two such cases are discussed in Examples \ref{example:spherical_strip} and \ref{example:spherical_cap}. Through these examples, we also see that it is equally possible to cook up examples of constant scalar curvature metrics where $\scal_g/(n-1)\in \sigma_N(\Delta)$. However, the assumption that this is not the case is generic enough that we may assume it for most of the following section. 

\section{Linearized Scalar Curvature}

Since we shall make excessive use of it, we shall shorten the notation for the linearised scalar curvature. Let $h\in S^2M$ and set
\begin{align*} P_gh:&=D_g\scal(h)=\Delta_g\trace_gh-\g{\Ric_g}{h}_g+\delta_g\delta_g h\\
P_g:&\;C^{\infty}(S^2M)\longrightarrow C^\infty(M).\end{align*}
Equally important is the formal adjoint of $P_g$,
\begin{align*} P_g^*\psi:&=\Delta_g\psi\cdot g-\psi\Ric_g+ \nabla_g^2\psi\\
P_g^*:&\;C^\infty(M)\longrightarrow C^\infty(S^2M).\end{align*}
Whenever the metric is clear from context, we shall simply write $P$ and $P^*$.

\subsection{Properties}

\begin{lemma}\label{properties} Let $(M,g)$ be a Riemannian manifold and $\psi$ a function, then
\begin{enumerate}
    \item $P_g\p{\psi g}=\g{P^*_g\psi}{g}$.
\end{enumerate}
If $g$ is of constant scalar curvature,
\begin{enumerate}[resume]
    \item $P_g\circ \delta^*_g=0$,
    \item $\delta_g \circ P^*_g=0$.
\end{enumerate}
If $g$ is Einstein,
\begin{enumerate}[resume]
    \item $P_g\circ \Delta_L=\Delta_g\circ P_g $,
    \item $\Delta_L\circ P^*_g=P^*_g\circ \Delta_g$,
\end{enumerate}
where $\Delta_Lh=\nabla^*\nabla h+\Ric_g\times h-2\overset{\circ}Rh$ is the Lichnerowicz Laplacian on $S^2M$.
\end{lemma}
\begin{proof} (\textit{1.}) is a direct comparison 
$$P\p{\psi\cdot g}=(n-1)\Delta \psi-\scal\psi=\tr\p{\Delta\psi\cdot g-\psi\Ric+\nabla^2\psi}=\tr P^*\psi.$$
(\textit{2.}) follows from constancy of $\scal_g$ and
$$P\p{\delta^*\alpha}=\frac{1}{2}\g{d\scal_g}{\alpha}=\frac{1}{2}\mathcal{L}_{\alpha^\sharp}\scal_g.$$
(\textit{3.}) is a consequence of a Ricci identity;
$$\Delta\nabla \psi=\nabla\Delta \psi-\Ric\cdot \nabla \psi,$$
applied to 
$$\delta\p{\Delta\psi \cdot g-\psi\Ric +\nabla^2\psi}=-\nabla\Delta\psi +\Ric\cdot \nabla \psi-\psi\delta\Ric+\Delta\nabla \psi,$$
followed by the use of the contracted Bianchi identity, $\delta\Ric=-\frac{1}{2}d\scal=0$. \\

(\textit{4.}) and (\textit{5.}) uses the naturality of the Lichnerowicz Laplacian, $\Delta_L$, and the Hodge Laplacian on one-forms, $\Delta_H$;
$$\tr \Delta_L h=\Delta \tr h, \qquad \delta\delta\Delta_Lh=\delta\Delta_H\delta h=\Delta\delta \delta h,\qquad \Delta_L\p{\psi\cdot g}=\Delta\psi \cdot g,\qquad \Delta_L\p{\delta^*d\psi}=\delta^*\Delta_Hd\psi=\delta^*d\Delta\psi,$$
with the added knowledge that $\delta^*d\psi=\nabla^2\psi$.
\end{proof}

The operators $P_g$ and $P_g^*$ are, á priori, only formally adjoint. The following Green identity gives conditions for proper adjointness. 

\begin{proposition} Let $(M,g)$ be a compact Riemannian manifold with closed boundary, let $\psi \in C^\infty(M)$ and $h\in S^2M$. Then
     \begin{equation}\label{PP*Green}
        (P^*_g\psi,h)_{L^2}-\p{\psi,P_gh}_{L^2}=\int_{\partial M}\psi\p{2\MC_g'(h)+\big\langle h^T,\SF_g\big\rangle}-\trgt {h}\cdot \nu\p{\psi}\; dA.
        \end{equation}
    In particular, if $h\in T_g\mathcal{M}_B$ and $\psi\in C^{\infty}_N(M)$
    $$(P^*_g\psi,h)_{L^2}=\p{\psi,P_gh}_{L^2}.$$
\end{proposition}
\begin{proof}
    Performing straightforward integration by parts twice, we obtain
    \begin{align*}\p{P^*\psi,h}_{L^2}=& \int_M \g{\Delta\psi \cdot g-\psi \Ric_g+\nabla^2\psi}{h}\; dV\\=&\int_M\g{\nabla\psi}{\nabla\tr h+\delta h}-\psi\g{\Ric_g}{h} \; dV+\int_{\partial M}h\p{\nabla \psi ,\nu}-\tr h\nu(\psi)\; dA\\ =& \int_M \psi\p{\Delta \tr h-\g{\Ric_g}{h}+\delta\delta h}\; dV+\int_{\partial M}\psi\p{\g{\delta h}{\nu}+\nu\p{\tr h}}+h\p{\nabla \psi,\nu}-\tr h\; \nu\p{\psi}\; dA.\end{align*}
    Since the boundary is closed, integration by parts on the boundary yields 
    \begin{align*}\int_{\partial M}h\p{\nabla \psi,\nu}\; dA=&\int_{\partial M}h\p{\nabla^T\psi,\nu}+\psi h(\nu,\nu)\; dA\\=&\int_{\partial M}\psi\p{\delta_{g^T}\p{h\cdot \nu}+h(\nu,\nu)}\; dA.\end{align*}
    The Green formula now follows from equation \eqref{BC2nongeometric}. 
\end{proof}

The round sphere plays an important role in the study of closed compact Einstein manifolds. When a boundary is present, the analogous Einstein manifold is the hemisphere $\mathbb{S}^n_+$ with the standard round metric $g_{\mathbb{S}^n}$. This was already seen when we discussed Escobar's eigenvalue estimate, but is most clearly seen in the rigidity statement of the following theorem.

\begin{theorem}[\cite{Esc}, Theorem 4.2] \label{Escobar4.2} Let $(M,g)$ be a compact Einstein manifold with boundary. Then 
$$\ker P^*_g|_{C_{g,N}^\infty (M)}\neq \tb{0} $$
if and only if $(M,g)$ is isometric to $\p{\mathbb{S}^n_+(r),g_{\mathbb{S}^n}}$, with $r=\sqrt{\frac{n(n-1)}{\scal_g}}$.
    
\end{theorem}

\begin{remark} In case $(M,g)=\p{\mathbb{S}^n_+(r),g_{\mathbb{S}^n}}$, where $\mathbb{S}^n_+(r)$ is identified with 
$$\tb{x=(x_1,\ldots ,x_{n+1})\in \R^{n+1}:\abs{x}=r }\cap\tb{x_{n+1}\geq 0}, $$
the kernel of $P^*_g|_{C_{g,N}^\infty}$ is the $n$-dimensional space spanned by the usual conformal vector fields on the (full) sphere, excluding the field generated by the coordinate function $x_{n+1}$. To see why this is, suppose $\psi\in C_{g,N}^\infty$ satisfies 
$$0=P^*_g\psi=\p{\Delta\psi-(n-1)\kappa\psi}\cdot g+\nabla^2_g\psi,$$
where $\kappa=r^{-2}$ is the sectional curvature. Since $\psi$ is Neumann, $P^*_g$ restricts to the corresponding operator on $\partial M=\mathbb{S}^{n-1}$:
$$0=\p{P^*_g\psi}^T=P_{g^T}^*\psi=\p{\Delta_{g^T}\psi-(n-2)\kappa\psi}\cdot g^T+\nabla^2_{g^T}\psi.$$
It is well known that the kernel of $P^*_g$ on the $(n-1)$-sphere is spanned by the $n$ conformal gradient vector fields generated by the coordinate fields in $\R^n\simeq\tb{x_{n+1}=0}$, see \cite{Oba}. Each of these fields can be realised as the restriction of the corresponding field of $\mathbb{S}^n_+$ in $\R^{n+1}$ (fixing the $(n+1)$'st coordinate). Furthermore, if $\psi_1,\psi_2\in \ker P^*|_{C^{\infty}_{g,N}}$ restricts to the same function on $\partial M$, then $\overline{\psi}=\psi_1-\psi_2\in\ker P^*|_{C^{\infty}_{g,N}}$. However, as Escobar shows in the proof of the theorem above, such functions must take their minimum and maximum values on $\partial M$. Since $\overline{\psi}|_{\PS}=0$, it must be trivial. We have thereby shown that there is a one-to-one correspondence between the $n$-dimensional kernel of $P^*_{g^T}$ and $\ker P^*_g|_{C^{\infty}_{g,N}}$.
\end{remark}

\subsection{A Slice Theorem}

In this section we seek a decomposition of the space of deformations of a metric $g\in \mathcal{M}_B$, akin to the one obtained in \cite{Koi} for a closed manifold. We have already mentioned that the space of Neumann multiples of $g$ serves as the tangent space for the subspace $[g]\subset \mathcal{M}_B$. The following proposition establishes a decomposition of the entire tangent space $T_g\mathcal{M}_B$, in the case that $g$ has constant scalar curvature, $\scal_g\equiv c$. From now on, we shall refer to such metrics as CSC.

\begin{proposition}\label{prop:split}
   \begin{enumerate}
       \item  Let $\p{M,g}$ be a CSC metric on a compact manifold with boundary, such that $\frac{\scal_g}{n-1}\notin \sigma_N(\Delta)$, then
       \begin{equation}\label{CSCsplitting}T_g\mathcal{M}_{B}=C^{\infty}_{N}(M)\cdot g\oplus \ker_g \p{\Delta_g\circ P_g}=C_{g,N}^\infty(M)\cdot g\oplus \ker\p{\Delta_g\circ P_g}.\end{equation}
    \item If, additionally, $g$ is Einstein, then $\ker_g\p{\Delta_g\circ P_g}=\ker_gP_g$ and
    \begin{equation}\label{Esplitting} T_g\mathcal{M}_B=C^{\infty}_{N}(M)\cdot g\oplus \ker_g P_g=\R\cdot g\oplus C_{g,N}^\infty(M)\cdot g\oplus \ker_g P_g.\end{equation}  
   \end{enumerate}
   In all cases $\ker_g$ denotes the kernel in the subspace of tensors with zero total trace.
\end{proposition}

\begin{proof}
    Let $g$ be CSC with $\scal_g\neq 0$, and take an arbitrary $h\in T_g\mathcal{M}_B$. Consider the elliptic boundary value problem
    \begin{equation}\label{SplittingBVP}\begin{cases}
        \p{\Delta -\frac{\scal_g}{n-1}}\psi=\frac{1}{n-1}Ph&\text{in}\; M\\
        \nu\p{\psi}=0 &\text{on}\; \partial M
    \end{cases}\end{equation}
    Since $\scal_g/(n-1)\notin \sigma_N(\Delta)$, the Fredholm alternative ensures a unique solution $\psi\in C_N^\infty(M)$ of \eqref{SplittingBVP}. However, $\psi$ will not have vanishing mean unless $\p{\Ric_g,h}_{L^2}=0$. Setting $c:=\frac{1}{\abs{M}}\int_M\psi\;dV$, where $\abs{M}$ denotes the volume of $M$, we construct a unique normalised solution 
    $$\widetilde{\psi}:=\psi-c \in C_{g,N}^\infty(M).$$
    By defining $k:=h-\widetilde{\psi}\cdot g$, we have obtained a decomposition $h=\widetilde{\psi}\cdot g+k$, where 
    \begin{equation}\label{eq:Pk}\begin{split}Pk&=Ph-P\p{(\psi-c)\cdot g}\\ &=Ph-(n-1)\p{\Delta -\frac{\scal_g}{n-1}}(\psi-c)\\ &=-c\scal_g.\end{split}\end{equation}
    As the right hand side is constant, $k$ is certainly in the kernel of $\Delta \circ P$. Since both $h$ and $\widetilde{\psi}\cdot g$ are in $T_g\mathcal{M}_B$, so is $k$.\\

    In the special case where $\scal_g=0$, \eqref{SplittingBVP} has a solution if and only if $\p{\Ric_g,h}_{L^2}=0$, so we will have to modify the problem slightly:
    \begin{equation}\label{SplittingBVP2}\begin{cases}
        \Delta\psi=\frac{1}{n-1}\p{Ph+\frac{1}{\abs{M}}\p{\Ric_g,h}_{L^2}}&\text{in}\; M\\
         \nu\p{\psi}=0 &\text{on}\; \partial M.
    \end{cases}\end{equation}
    As before, we have a unique normalised solution $\widetilde{\psi}$ of \eqref{SplittingBVP2}, and 
    \begin{align*}\Delta \p{P\p{h-\widetilde{\psi}\cdot g}}&=\Delta\p{Ph-(n-1)\Delta\psi}\\&=-\Delta\p{\frac{1}{\abs{M}}(\Ric_g,h)_{L^2}}=0.\end{align*}

    When $g$ is Einstein, we shall assume $\int_M\tr h\; dV=0$. This can be made rigorous by subtracting $\big(\frac{1}{n\abs{M}}\int_M\tr h\; dV\big)g$, which can be added to the $C_{g,N}^\infty(M)\cdot g$ component of the splitting \eqref{Esplitting}. If the Einstein constant is non-zero, $\mu \neq 0$, the unique solution $\psi$ of \eqref{SplittingBVP} will have zero mean. This means $c=0$ in \eqref{eq:Pk} and $k\in \ker P$. Furthermore, as both constituents of $k$ have vanishing total trace, we indeed have $k\in \ker_gP$.\\
    
    If $g$ is Ricci-flat, the boundary value problems \eqref{SplittingBVP} and \eqref{SplittingBVP2} coincide. However, the unique solution will generally not have vanishing mean. Normalising does not pose an issue this time, as \eqref{eq:Pk} vanishes even when $c\neq 0$. 
\end{proof}

The decomposition \eqref{CSCsplitting} shows that it is sufficient that
$$S'_g(h)=0\qquad \forall h\in \ker_g\p{\Delta_g\circ P_g},$$
for a smooth CSC metric to be Einstein. This mirrors the situation in which the manifold is closed. The following remark explores a way in which the closed case is different.

\begin{remark}\label{noboundsplitting}
    Let $(M,g)$ be a closed compact Einstein manifold (without boundary), then it holds under no additional assumptions, that
    \begin{equation}\label{P*Pdecomp}T_g\mathcal{M}=\Imag P^*|_{C^\infty(M)}\oplus_{L^2} \ker_g P.\end{equation}
    This may even be refined to 
    \begin{equation}\label{closed_splitting}T_g\mathcal{M}= \Imag P^*|_{C^\infty(M)}\oplus TT_g\oplus \Imag \delta^*,\end{equation}
    which uses the standard splitting 
    $$S^2M=\ker \delta\oplus\Imag \delta^*,$$
    the inclusions $\Imag P^*\subset \ker\delta$ and $\Imag \delta^*\subset \ker_g P$, as well as 
    \begin{equation*}\ker_g P\cap \ker \delta \subset \ker_g \p{\p{\Delta -\mu}\tr}=\ker \tr.\end{equation*}
    If $\psi\in C^\infty(M)$, then 
    $$P^*\psi=\p{\Delta-\mu}\psi\cdot g+\nabla^2\psi\in C^\infty(M)\cdot g+\Imag \delta^*,$$
    which means \eqref{closed_splitting} can also be written as 
    $$T_g\mathcal{M}= \p{C^\infty(M)\cdot g+\Imag \delta^*}\oplus TT_g.$$
    The sum is direct if and only if $\ker P^*|_{C^\infty_g(M)}=\tb{0}$, which is the case if and only if $\p{M,g}$ is not isometric to the round sphere.\\

    While much of this remain true in the presence of a boundary, given an appropriate choice of boundary conditions, there are enough differences that we can not achieve a similar decomposition. The problem is mostly with how the conditions interact with the different operators. For example, though $\Imag P^*|_{C^\infty_N(M)}$ and $\ker_gP\subset T_g\mathcal{M}_B$ are $L^2$-orthogonal, there is no chance that \eqref{P*Pdecomp} holds, as $P^*$ simply does not map into $T_g\mathcal{M}_B$. Even more troubling is the fact that the equality $\ker_g\p{(\Delta-\mu)\tr}=\ker\tr$ is not true in $T_g\mathcal{M}_B$. As such, there is no natural way to obtain $TT_g$ as a factor in a decomposition of $T_g\mathcal{M}_B$.
\end{remark}

Now follows a boundary conditioned version of a well-known slice theorem by Koiso \cite{Koi}. See also \cite[Theorem 4.44]{Bes}.

\begin{theorem}\label{MB_cor} 
    Consider the space $\mathcal{M}^{k,\alpha}_B$ from \eqref{MB} and define the subspace of constant scalar curvature (CSC) metrics;
    $$\mathcal{C}_B:=\tb{g\in \mathcal{M}^{k,\alpha}_B\;\; : \;\;\scal_g \text{ is constant}}.$$
    Suppose $g\in \mathcal{C}_B$ is Einstein and $\scal_g/\p{n-1}\notin \sigma_N(\Delta)$. Then $\mathcal{C}_B$ is a submanifold of $\mathcal{M}_B^{k,\alpha}$ in a neighbourhood of $g$, and 
    $$T_g\mathcal{C}_B=\ker \p{\Delta_g\circ P_g}=\R \cdot g\oplus \ker_g\p{\Delta_g\circ P_g},$$
    and the map 
    $$(f,\widetilde{g})\mapsto (1+f)\cdot \widetilde{g},\qquad \qquad C^{k,\alpha}_{g,N}(M)\times \mathcal{C}_B\to \mathcal{M}_B^{k,\alpha}$$
    is a local diffeomorphism from a neighbourhood of $(0,g)$ to a neighbourhood of $g$.
\end{theorem}
\begin{proof}
    Consider the map
    \begin{align*}\Phi(g)&:=\scal_g-\frac{1}{\abs{M}_g}\int_M \scal_g\;dV_g\\
    \Phi &: \mathcal{M}_B^{k,\alpha}\longrightarrow C^{k-2,\alpha}_g(M),\end{align*}
    where $\abs{M}_g=\int_M1\;dV_g$. It is important to note that, by design, $\int_M\Phi(g)\; dV_g=0$. This means that the linearisation at $g\in \mathcal{C}_B$, is a map $D_g\Phi:T_{g}\mathcal{M}^{k,\alpha}_B\to T_gC^{k-2,\alpha}_g(M)=C_g^{k-2,\alpha}(M)$. According to \eqref{Esplitting}, domain of $D_g\Phi$ splits as 
    $$T_g\mathcal{M}_B=C^{k,\alpha}_{g,N}(M)\cdot g \oplus \R\cdot g\oplus \ker_gP.$$
    If $h\in \ker_gP$, then $D_g\abs{M}(h)=0$ and 
    $$D_g\Phi(h)=Ph-\frac{1}{\abs{M}_g}\int_M Ph+\frac{1}{2}\scal_g\tr h\; dV=0.$$
    For $h=c\cdot g\in \R\cdot g$, we have $D_g\abs{M}(h)=\frac{nc}{2}\abs{M}_g$ and $Ph=-nc\mu$, where $\mu$ is the Einstein constant of $g$. Thus 
    \begin{align*}D_g\Phi(h)&=Ph-\frac{1}{\abs{M}_g}\int_M Ph\; dV+\frac{D_g\abs{M}_g(h)}{\abs{M}_g^2}\int_M \scal_g\;dV_g-\frac{1}{\abs{M}_g}\int_M \frac{1}{2}\scal_g\tr h\;dV_g\\
    &=-nc\mu+\frac{1}{\abs{M}_g}\int_M nc\mu\; dV+\frac{nc}{2\abs{M}_g}\int_M\scal_g\; dV-\frac{1}{\abs{M}_g}\int_M\frac{nc}{2}\scal_g\; dV=0.\end{align*}
    On the volume preserving complement $C_{g,N}^{k,\alpha}(M)\cdot g$, the operator $D_g\Phi$ reduces to 
    $$D_g\Phi\p{\psi\cdot g}=P\p{\psi\cdot g}=(n-1)\p{\Delta-\frac{n\mu}{n-1}}\psi.$$
    By assumption, $\Delta-\frac{n\mu}{n-1}:C^{k,\alpha}_{g,N}(M)\to C^{k-2,\alpha}_g(M)$ is invertible. This means $D_g\Phi$ is surjective to $C_g^{k-2,\alpha}(M)$, and by the implicit function theorem for Banach spaces (e.g. \cite[Theorem 3.12]{FNSS}), $\ker\Phi=\mathcal{C}_B\subset \mathcal{M}_B^{k,\alpha}$ is locally a real analytic manifold with tangent space 
    $$T_g\mathcal{C}_B=\ker D_g\Phi=\R\cdot g\oplus \ker_gP.$$
    Since $D_g\Phi$ has complemented kernel by \eqref{Esplitting}:
    $$T_g\mathcal{M}_B^{k,\alpha}=C^{k,\alpha}_{g,N}(M)\cdot g\oplus \ker D_g\Phi,$$
    we may apply the inverse function theorem to obtain a local diffeomorphism
    $$C^{k,\alpha}_{g,N}(M)\times \ker\Phi=C^{k,\alpha}_{g,N}(M)\times \mathcal{C}_B\longrightarrow \mathcal{M}_B^{k,\alpha},\qquad (f,\widetilde{g})\longmapsto (1+f)\widetilde{g}.$$
\end{proof}

\begin{remark}
    Using Theorem \ref{Escobar4.2}, it is possible to relax the condition $\scal_g/(n-1)\notin \sigma_N(\Delta)$ to non-isometry with $\mathbb{S}^n_+$ for the first part of the theorem, but it is necessary in the second. 
\end{remark}

Unfortunately, the boundary conditions does not allow for a slice theorem near an arbitrary (non-Einstein) CSC metric, along the lines of the splitting \eqref{CSCsplitting}. The problem is that the conditions on $h$ at the boundary, does not transfer to elliptic conditions on $Ph$, such as Neumann or Robin. As a result, $\Delta Ph=0$ does not imply $Ph\equiv c$. Otherwise, the result would follow by a similar proof.\\

We conclude this section with some properties of $\mathcal{M}_B$. For now, we are mostly interested in conformally invariant properties. One such property is convexity of the boundary. This can be seen by calculating a conformal transformation formula for the second fundamental form:
$$\SF_{e^{2f}g}=\frac{1}{2}\big(\mathcal{L}_{e^{-f}\nu}\p{e^{2f}g}\big)^T=\frac{1}{2}e^{-f}\p{e^{2f}\mathcal{L}_\nu g+\nu\p{e^{2f}}\cdot g}^T=e^{f}\p{\SF_g+\nu(f)g^T},$$
which shows that Neumann conformal factors preserve positivity of the second fundamental form. This is particular important with regards to stability, as we have the lower bound \eqref{Obata} on the principal Neumann eigenvalue whenever the boundary is convex. Similarly, scalar curvature positivity is shared among conformally related CSC metrics in $\mathcal{M}_B$.

\begin{proposition}\label{properties_of_MB} Let $g_1,g_2\in \mathcal{M}_B$ be constant scalar curvature metrics in the same conformal class, $[g_1]=[g_2]$. Then $\scal_{g_1}$ and $\scal_{g_2}$ are either both positive, both negative, or both zero.

\end{proposition}
\begin{proof}
    Since rescalings preserve the sign of the scalar curvature, we will assume that $g_2$ is scaled such that $\mathrm{Vol}_{g_2}(M)=1$. Let $g_2=u^{\frac{4}{n-2}}g_1$ for a positive Neumann function $u\in C_N^\infty(M)$. Using the conformal identities,
    $$\scal_{g_2}=u^{-\frac{4}{n-2}}\p{\scal_{g_1}+\frac{4(n-1)}{n-2}\frac{\Delta_{g_1}u}{u}}$$
    and
    $$dV_{g_2}=u^{\frac{2n}{n-2}}\; dV_{g_1},$$
    we find that 
    \begin{align*}\scal_{g_2}=&\int_M \scal_{g_2}\; dV_{g_2}=\int_M\frac{4(n-1)}{n-2}\abs{\nabla u}_{g_1}^2+\scal_{g_1} u^2\; dV_{g_1}.\end{align*}
    If $\scal_{g_1}>0$, the right hand side is indisputably positive. If $\scal_{g_1}=0$, the right hand side is positive if $u$ non-constant. However, the previous argument can be used to show that $\scal_{g_2}>0$ implies $\scal_{g_1}>0$. It is therefore not possible that only one scalar curvature vanishes. The statement for the case of negative scalar curvature follows by exclusion.
\end{proof}

\section{Stability of Einstein Metrics}

Proposition \ref{prop:split} and the slice theorem, Theorem \ref{MB_cor}, are important requisites for studying the extremality properties of Einstein metrics as critical points for the Einstein--Hilbert action. Specifically, we seek to discover in which directions Einstein metrics are local minimisers/maximisers.

\subsection{The Stability Operator}

As was the case when we introduced $P_g=D_g\scal$, we formally define the stability operator on the bundle $S^2M$. In practice, the operator acts on sections of a given regularity, determined by the regularity we consider on the space of metrics. Assume smoothness unless stated otherwise.

\begin{definition} Let $(M,g)$ be an Einstein manifold with $\Ric_g=\mu g$. For $h\in S^2M$ we define the stability operator
    \begin{equation*}\begin{split}F_gh:=&2D_g\Ric(h)-D_g\scal(h)\cdot g-2\mu h \\=& \Delta_L h-2\delta^*_g\beta_g h-\Delta_g \mathrm{tr}_g h\cdot g+\mu\mathrm{tr}_g h\cdot g-\delta_g\delta_g h\cdot g-2\mu h\\ =&\Delta_E h-2\delta^*_g\beta_g h-P_gh\cdot g\end{split}\end{equation*}
where $\Delta_Eh=\nabla^*\nabla h-2\overset{\circ}{R}h$ is the Einstein operator.
\end{definition}

The operator may be familiar to most readers as the (formal) $L^2$-Hessian of the Einstein--Hilbert action;
$$S''_g(h)=-\frac{1}2\Big(F_g h,h\Big)_{L^2}.$$
In contrast to the fact that a CSC metric is Einstein if and only if $S_g'(h)=0$ for all $h\in \ker\p{ \Delta\circ P}$, we shall see that instabilities are not always confined to this space. In Examples \ref{example:spherical_strip} and \ref{example:spherical_cap} of the next section, we discuss two such cases.

\begin{lemma} \label{F_properties}
    The stability operator, $F_g:S^2M\to S^2M$, satisfies
    \begin{enumerate}
        \item $F_g\circ \delta^*_g=0$,
        \item $\delta_g \circ F_g=0$,
        \item $\mathrm{tr}_g F_g=-(n-2)P_g$,
        \item $F_g\p{\psi g}=-(n-2)P^*_g\psi$,
        \item $P_g\circ F_g=\p{\Delta_g-2\mu-\mathrm{tr}_g P^*_g}\circ P_g$.
    \end{enumerate}
    Additionally, the decompositions of Proposition \ref{prop:split} are orthogonal with respect to $F_g$.
\end{lemma}
\begin{proof}
   (\textit{1.}) follows from the diffeomorphism invariance of the Einstein--Hilbert action, but may also be shown directly: Let $\omega\in \Omega^1(M)$ and apply the second identity of Lemma \ref{properties} ($P\circ \delta^*=0$);
    \begin{align*}F\p{\delta^*\omega}=&\Delta_L\delta^*\omega-2\delta^*\beta \delta^* \omega-2\mu \delta^*\omega\\
    =&\delta^*\p{\nabla^*\nabla \omega-\mu \omega-2\beta\delta^*\omega},\end{align*}
    in which we also used the commutation relation $\Delta_L\circ \delta^*=\delta^*\circ \Delta_H$, where $\Delta_H\omega =\nabla^*\nabla\omega+\Ric\cdot \omega=\nabla^*\nabla \omega+\mu\omega $ is the Hodge Laplacian on one-forms.
    A standard application of a Ricci identity shows that 
    $$2\beta \delta^*\omega=\nabla^*\nabla \omega-\mu \omega,$$
    which proves the claim.\\
    
    (\textit{2.}) may be proven in several ways. Writing $g_t:=g+th$ it becomes a direct consequence of the contracted Bianchi identity: 
    \begin{align*}0=&\left.\frac{d}{dt}\right|_{t=0}\delta_{g_t}\p{2\Ric_{g_t}-\scal_{g_t}\cdot g_t-2\mu g_t}\\=&\delta Fh-\delta \p{2\Ric_g-2\mu g}\\ =&\delta Fh.\end{align*}
    It will also follow as a corollary to the Green identity \eqref{greens} below.\\

    (\textit{3.}) follows from the observation that 
    $$\left.\frac{d}{dt}\right|_{t=0}\tr_{g_t} \p{2\Ric_{g_t}-\scal_{g_t}\cdot g_t-2\mu g_t}=-(n-2)\left.\frac{d}{dt}\right|_{t=0}\scal_{g_t}=-(n-2)Ph,$$
    while at the same time, using the Leibniz rule $\p{\tr k}'=\tr k'-\g{h}{k}$,
    \begin{align*}\left.\frac{d}{dt}\right|_{t=0}\tr_{g_t} \p{2\Ric_{g_t}-\scal_{g_t}\cdot g-2\mu g_t}=&\tr \p{Fh-\scal_g \cdot h}-\g{h}{2\Ric_g-\scal_g\cdot g -2\mu g}\\ =&\tr Fh-\scal_g\tr h+\scal_g \tr h\\=& \tr Fh.\end{align*}

    (\textit{4.}) is a direct computation, using 
    $$\Delta_L(\psi\cdot g)-\Delta\tr(\psi\cdot g)\cdot g-\delta\delta (\psi\cdot g)=-(n-2)\Delta\psi \cdot g$$
    and 
    $$2\delta^*\beta (\psi\cdot g)=\p{n-2}\nabla^2\psi.$$
    (\textit{4.}) also proves the orthogonality statement for the splittings, when it is combined with the Green identity for $P$ and $P^*$, equation \eqref{PP*Green}.\\

    (\textit{5.}) follows from equations (\textit{1.}), (\textit{2.}) and (\textit{5.}) of Lemma \ref{properties};
    \begin{align*}P\p{Fh}&=P\p{\Delta_Lh-2\mu h-2\delta^*\beta h-Ph\cdot g}\\&=P\p{\Delta_L h}-2\mu Ph-P\p{Ph\cdot g}\\ &=\Delta Ph-2\mu Ph-\tr P^* Ph.\end{align*}
\end{proof}

In the end, we are looking to set up a well defined elliptic eigenvalue problem for the stability operator. The first point of Lemma \ref{F_properties} shows that $F_g$ is not elliptic on its own, as it has an infinite-dimensional kernel. This problem will be addressed in Section \ref{sec:stab}. For now, we note that ellipticity is most useful when it is accompanied by self-adjointness. This we address with the following Green identity for $F_g$.

\begin{proposition}
Let $g$ be Einstein and $h,k\in S^2M$, then 
\begin{equation}\label{greens}\begin{split}
    \int_M\g{F_gh}k-\g{h}{F_gk}\; dV=-&\int_{\partial M}\g{2\SF'_g(h)-\frac{1}{n-1}\trgt h\; \SF_g}{\overset{\circ}{k}\!\:^T}-\g{\overset{\circ}{h}\!\:^T}{2\SF'_g(k)-\frac{1}{n-1}\trgt k\; \SF_g}\\ &+\frac{n-2}{n-1}\trgt h\p{2\MC'_g(k)+\g{k^T}{\SF_g}}-\frac{n-2}{n-1}\trgt k\p{2\MC'_g(h)+\g{h^T}{\SF_g}}\; dA.
\end{split}\end{equation}
In particular, $F_g$ is self-adjoint on $T_g\mathcal{M}_B$.
\end{proposition}

\begin{proof}
    We commence by calculating a Green identity for $Ph\cdot g$. For arbitrary $h,k\in S^2M$, integration by parts yields
    \begin{equation}\label{PGreens}\begin{split}\int_M\g{Ph\cdot g}{k}-\g{h}{Pk\cdot g}\; dV=&\int_M\tr k\p{\Delta\tr h+\delta \delta h}-\tr h\p{\Delta\tr k+\delta \delta k}\; dV\\=&\int_M\g{\delta h}{d \tr k}-\g{d \tr h}{\delta k}\; dV\\&-\int_{\partial M}\tr k\g{\delta h+d \tr h}{\nu}-\tr h\g{\delta k+d \tr k}{\nu}\; dA\\ =&\int_M\g{h}{\nabla^2\tr k}-\g{\nabla^2\tr h}{k}\; dV\\&-\int_{\partial M}h\p{d\tr k,\nu}-k\p{d\tr h,\nu}-\tr k\;\delta_{g^T}\p{h\cdot \nu}+\tr h\;\delta_{g^T}\p{k\cdot \nu}\; dA\\&-\int_{\partial M}\tr k\p{2\MC'(h)+\g{h^T}{\SF}}-\tr h\p{2\MC'(k)+\g{k^T}{\SF}}\;dA,\end{split}\end{equation}
    where we used \eqref{BC2nongeometric}:
    $$\g{\delta h+d\tr h}{\nu}=2\MC'(h)+\g{h}{\SF}-\delta_{g^T}\p{h\cdot \nu}.$$
    Performing integration by parts in the first boundary integral in \eqref{PGreens}, we obtain
     \begin{equation}\begin{split}\int_M\g{Ph\cdot g}{k}-\g{h}{Pk\cdot g}\; dV=&\int_M\g{h}{\nabla^2\tr k}-\g{\nabla^2\tr h}{k}\; dV\\&-\int_{\partial M}h\p{\nu,\nu}\g{d\tr k}{\nu}-k\p{\nu,\nu}\g{d\tr h}{\nu} dA\\&-\int_{\partial M}\tr k\p{2\MC'(h)+\g{h^T}{\SF}}-\tr h\p{2\MC'(k)+\g{k^T}{\SF}}\;dA.\end{split}\end{equation}
     We now calculate a similar formula for the operator $Qh:=\nabla^*\nabla h-2\delta^*\delta h$;
     \begin{equation}\label{QGreens}\begin{split}\int_M\g{Qh}{k}-\g{h}{Qk}\; dV=&-\int_{\partial M}\g{\nabla_\nu h}{k}-\g{h}{\nabla_\nu k}+2k\p{\delta h,\nu}-2h\p{\delta k,\nu}\;dA\\=&-\int_{\partial M}\g{\nabla_\nu h^T-2\delta_{g^T}^*\p{h\cdot \nu}}{k^T}-\g{h^T}{\nabla_\nu k^T-2\delta_{g^T}^*\p{k\cdot \nu}}\; dA\\ &-\int_{\partial M}k(\nu,\nu)\p{2\g{\delta h}\nu+\nabla_{\nu}h(\nu,\nu)}-h(\nu,\nu)\p{2\g{\delta k}{\nu}+\nabla_{\nu}k(\nu,\nu)}\; dA.\end{split}\end{equation}
     Using the formula for the first variation of the second fundamental form, \eqref{eq:SFvariation}, we get 
     \begin{align*}\int_{\partial M}&\g{2\SF'(h)}{k^T}-\g{h^T}{2\SF'(k)}\;dA\\&=\int_{\partial M}\g{\nabla_\nu h^T-2\delta_{g^T}^*\p{h\cdot \nu}-h(\nu,\nu)\SF}{k^T}-\g{h^T}{\nabla_\nu k^T-2\delta_{g^T}^*\p{k\cdot \nu}-k(\nu,\nu)\SF}\; dA.\end{align*}
     This may be used in \eqref{QGreens} to obtain
      \begin{equation*}\begin{split}\int_M\g{Qh}{k}-\g{h}{Qk}\; dV=&-\int_{\partial M}\g{2\SF'(h)}{k^T}-\g{h^T}{2\SF'(k)}\; dA\\ &-\int_{\partial M}k(\nu,\nu)\p{\g{\delta h}\nu+\delta_{g^T}(h\cdot \nu)}-h(\nu,\nu)\p{\g{\delta k}{\nu}+\delta_{g^T}\p{k\cdot \nu}}\; dA.\end{split}\end{equation*}
      Combining this with the Green identity for $Ph\cdot g$ yields
      \begin{equation*}\begin{split}\int_M\g{F_gh}{k}-\g{h}{F_gk}\; dV=&-\int_{\partial M}\g{2\SF'(h)}{k^T}-\g{h^T}{2\SF'(k)}\; dA\\ &-\int_{\partial M}\trgt h\p{2\MC'(k)+\g{k^T}{\SF}}-\trgt k\p{2\MC'(h)+\g{h^T}{\SF}}\; dA. \end{split}\end{equation*}
      The claimed identity now follows from  
      $$\trgt \SF'(h)=\MC'(h)+\g{h^T}{\SF}.$$
      Self-adjointness on $T_g\mathcal{M}_B$ is immediate, as $h\in T_g\mathcal{M}_B$ exactly when $\overset{\circ}{h}\!\:^T=0$ and $2\MC'(h)+\langle h^T,\SF\rangle=0$.
\end{proof}

\begin{remark}\label{infEinsteindef}
    A standard argument shows that the kernel of $F_g$ agrees with the kernel of the linearized Einstein tensor, modulo rescalings and diffeomorphisms. Indeed, by Lemma \ref{F_properties} point (\textit{3.})
    \begin{align*}E_g'(h)&=D_g\Ric(h)-\frac{1}{n}D_g\scal(h)\cdot g-\mu g\\ &=\frac{1}{2}F_gh+\frac{n-2}{2n}P_g h\cdot g\\&=\frac{1}2\overset{\circ}{F_gh}.\end{align*}
    So if $h\in \ker F_g$, it is evident that $E_g'(h)=0$. Conversely, if $E'_g(h)=0$, then $F_g h$ is pure trace;
    $$F_gh=\frac{1}{n}\tr F_g h\cdot g=-\frac{n-2}{n}P_g h\cdot g.$$
    However, as $F_g h$ is always transverse, we have 
    $$0=\delta_g\delta_g F_g h=-\frac{n-2}{n}\delta_g\delta_g\p{P_gh\cdot g}=\frac{n-2}{n}\Delta_g\p{P_g h},$$
    and Proposition \ref{prop:split} says that $\ker\p{\Delta_g \circ P_g}=\R\cdot g\oplus \ker_g P_g$, where $\ker_g P_g$ denotes the kernel elements with zero total trace. Consequently, either $h=c g$ is a rescaling, or $P_gh=0$ and
    $$F_gh=-\frac{n-2}{n}P_g h\cdot g=0.$$
    Thus $\ker_g F_g\subset T_g\mathcal{M}_B$ corresponds to the space of infinitesimal Einstein deformations, 
    $$\varepsilon(g):=\ker_g E'_g\cap \ker_g \delta_g\subset T_g\mathcal{M}_B$$ 
    (see e.g. \cite{Bes} Definition 12.29). To be even more precise, we can use apply \cite[Lemma 2.1]{An1} to show that 
    $$\ker_g F_g=\varepsilon(g)\oplus_{L^2}\Imag \delta^*|_{\Omega_0}.$$
\end{remark}

A direct corollary of Escobar's Theorem \ref{Escobar4.2}, using the fourth point of Lemma \ref{F_properties}, is the following proposition.  It can be regarded as an infinitesimal version of a later result, Theorem \ref{ObataTypeUniqueness}.

\begin{proposition}\label{infrig} Suppose $(M,g)$ is Einstein, but not isometric to $\mathbb{S}^n_+$, then 
$$\ker F_g\cap C_{g,N}^\infty(M)\cdot g=\tb{0}. $$
We shall say that $(M,g)$ is conformally non-degenerate.
\end{proposition}

It is important to observe that it does not imply that the quadratic  form associated with $F_g$ is non-vanishing on $C^\infty_{g,N}(M)\cdot g$. Specifically, even if $(M,g)$ is not a hemisphere, it is still possible that there exists $\psi\in C^\infty_{g,N}(M)$ with $\Delta_g\psi-\frac{n\mu}{n-1}\psi =0$. In that case
$$\p{F_g(\psi g),\psi g}_{L^2}=0\qquad \text{while }\qquad F(\psi g)\neq 0.$$
We shall return to this in Section \ref{sec:stab}.

\begin{proposition}[Constraint Conditions]\label{prop:hiddenBC} Let $h\in \ker_gP\subset T_g\mathcal{M}_B$ be a Bianchi-gauged eigenmode for $F_g$: 
$$F_gh=\lambda h\qquad \qquad \&\qquad \qquad \beta_g h=0,$$
for some $\lambda\in \R$. Denote the boundary conformal factor $\varphi:=\frac{1}{n-1}\trgt h$, then the following equations hold on $\partial M$,
\begin{align}\label{hidden1}\Delta_{g^T}\varphi-\mu\varphi-\frac{n-1}{n-2}\lambda \varphi &=-\frac{1}{n-2}\g{2\SF'_g(h)-\varphi\SF_g}{\tf{\SF}_g}-\frac{\lambda}{2(n-2)}\mathrm{tr}_g h,\\
\label{hidden2}\delta_{g^T}\p{2\SF'_g(h)-\varphi\SF_g}&=\p{n-2}\SF_g\cdot \nabla^T \varphi-\lambda(h\cdot \nu)^T.\end{align}
\end{proposition}
\begin{proof}
    Consider a test tensor $\beta^* \omega=\delta^*\omega+\frac{1}{2}\delta \omega \cdot g$ for some $\omega\in \Omega^1(M)$, not necessarily vanishing at the boundary. Then, since $Fh=\lambda h$ and $\beta h=0$,
    \begin{equation}\label{eq:constrproof1}\begin{split}
    \p{Fh,\beta^*\omega}_{L^2}&=\lambda\p{h,\beta^*\omega}_{L^2}\\&=\lambda\p{\beta h,\omega}_{L^2}+\lambda\int_{\partial M}\g{h}{\nu\otimes \omega}-\frac{1}{2}\tr h\g{\nu}{\omega}\; dA\\ &=\lambda\int_{\partial M}\g{h}{\nu\otimes \omega}-\frac{1}{2}\tr h\g{\nu}{\omega}\; dA.\end{split}\end{equation}
    An alternative way to evaluate $\p{Fh,\beta^*\omega}_{L^2}$, is to apply the Green identity \eqref{greens}. This yields
    \begin{equation}\label{eq:constrproof}\p{Fh,\beta^*\omega}_{L^2}=\p{h,F\beta^*\omega}_{L^2}-\int_{\partial M}\g{2\SF'(h)-\varphi \SF}{\overset{\circ}{\beta^*\omega}\!\:^T}+(n-2)\varphi\p{2\MC'(\beta^*\omega)+\g{\beta^*\omega^T}{\SF}}\; dA,\end{equation}
    where the remaining terms vanish as $h\in T_g\mathcal{M}_B$. From Lemma \ref{F_properties}, points (\textit{1.}) and (\textit{4.}), we find that 
    $$F\beta^*\omega=F\p{\delta^*\omega+\frac{1}{2}\delta\omega\cdot g}=\frac{1}{2}F(\delta \omega \cdot g)=-\frac{n-2}{2}P^*\delta \omega.$$
    This allows us to apply the Green identity \eqref{PP*Green} for $P$/$P^*$, which yields
    \begin{align*}\p{h,F\beta^*\omega}_{L^2}&=-\frac{n-2}{2}\p{h,P^*\delta\omega}_{L^2}\\ &=-\frac{n-2}{2}\p{Ph,\delta \omega}_{L^2}+\frac{n-2}{2}\int_{\PS}\trgt h\;\nu\p{\delta \omega}\; dA\\ &=\frac{(n-2)(n-1)}{2}\int_{\PS}\varphi\;\nu\p{\delta \omega}\; dA.\end{align*}
    This term has the effect of cancelling the only contribution from $\delta\omega\cdot g$ in the boundary integral in \eqref{eq:constrproof}. That is, since $\overset{\circ}{\beta^*\omega}\!\:^T=\overset{\circ}{\delta^*\omega}\!\:^T$ and 
    $$2\MC'\p{\beta^*\omega}+\g{\beta^*\omega}{\SF}=2\MC'\p{\delta^*\omega}+\g{\delta^*\omega}{\SF}+(n-1)\nu\p{\delta \omega},$$
    we have 
    \begin{align}\label{eq:constrproof2}
        \p{Fh,\beta^*\omega}_{L^2}=-\int_{\partial M}\Big\langle2\SF'(h)-\varphi \SF,\overset{\circ}{\delta^*\omega}\!\:^T\Big\rangle+(n-2)\varphi\p{2\MC'(\delta^*\omega)+\g{\delta^*\omega^T}{\SF}}\; dA.
    \end{align}
    Suppose now that $\omega|_{\PS}=f\nu$, for some $f\in C^\infty(\partial M)$, then $\delta^*\omega^T=f\SF$ and $2\MC'(\delta^*\omega)+\big\langle\delta^*\omega^T,\SF\big\rangle=\Delta_{g^T}f-\mu f$. By comparing \eqref{eq:constrproof1} and \eqref{eq:constrproof2}, we see that 
    \begin{equation*}
        \begin{split}
            \lambda\int_{\PS} f\p{h(\nu,\nu)-\frac{1}{2}\tr h}\; dA&=-\int_{\partial M}\gBig{2\SF'(h)-\varphi\SF}{\tf{\SF}}f+(n-2)\varphi\p{\Delta_{g^T}f-\mu f}\; dA\\ &=-\int_{\partial M}\gBig{2\SF'(h)-\varphi\SF}{\tf{\SF}}f+(n-2)\p{\Delta_{g^T}\varphi-\mu \varphi}f\; dA.
        \end{split}
    \end{equation*}
    Since $f$ is arbitrarily chosen, we must in fact have \eqref{hidden1} pointwise at $\partial M$.\\

    Using a second test tensor, $\beta^*\omega$ with $\omega|_{\partial M}=\omega^T$, the same comparison gives
    \begin{align*}
        \lambda\int_{\PS}\gBig{(h\cdot \nu)^T}{\omega}\; dA=&-\int_{\partial M}\gBig{2\SF'(h)-\varphi\SF}{\delta_{g^T}^*\omega}+(n-2)\varphi\p{\g{d^T \MC}{\omega}+\g{\SF}{\delta_{g^T}^*\omega}}\; dA\\ =&-\int_{\partial M}\gBig{\delta_{g^T}\p{2\SF'(h)-\varphi\SF}-(n-2)\SF\cdot \nabla^T\varphi}{\omega}  \; dA,  \end{align*}
where we have applied the momentum constraint equation $d^T \MC=-\delta_{g^T}\SF$. \eqref{hidden2} now follows.
    \end{proof}

It might not be immediately clear, how these equations may be used in practice. And indeed, we shall only have use for equation \eqref{hidden1} in the present paper (at the end of Section \ref{sec:SADS}). However, they serve to remind us that Einstein metrics carry additional structure on the boundary, the Einstein constraints, that impacts the eigenvalue problem.

\begin{remark} We refer to the \eqref{hidden1} and \eqref{hidden2} as constraint conditions, due to the fact that they represent the interplay between the boundary conditions and the Einstein constraint equations. Specifically, when $\lambda=0$, \eqref{hidden1} can be derived directly from the Hamiltonian constraint (contracted Gauss equation):
    $$\scal_{g^T}=(n-2)\mu+\MC^2-\abs{\SF}^2.$$
    By linearising each side in the direction of $h\in T_g\mathcal{M}_B$, we obtain
    \begin{align*}D_{g^T}\scal\!\big(\varphi g^T\big)&=2\MC\MC'(h)-2\g{\SF'(h)}{\SF}+\gbig{h^T}{\SF\times \SF}\\ &=\varphi\p{\abs{\SF}^2-\MC^2}-\g{2\SF'(h)-\varphi\SF}{\SF}.\end{align*}
    The left hand side can also be calculated, using standard variational formula for the scalar curvature in a conformal direction:
    \begin{align*}D_{g^T}\scal\!\pbig{\varphi g^T}&=(n-2)\Delta_{g^T}\varphi-\scal_{g^T}\varphi\\&=(n-2)\pbig{\Delta_{g^T}-\mu}\varphi+\p{\abs{\SF}^2-\MC^2}\varphi.\end{align*}
    Comparing the two formulae above and using $\trgt\p{2\SF'(h)-\varphi\SF}=0$, yields the $\lambda=0$ version of \eqref{hidden1}.\\

   Similarly, one obtains \eqref{hidden2} with $\lambda=0$ by linearising the momentum constraint (contracted Codazzi equation):
    \begin{align*}
        0&=\p{\delta_{g^T}\SF+d^T\MC }'(h)\\&=\delta_{g^T}'\SF+\delta_{g^T}\SF'(h)+d^T\MC'(h).
    \end{align*}
    If $h\in T_g\mathcal{M}_B$, we have 
    \begin{align*}d^T\MC'(h)&=-\frac{1}{2}d^T\p{\varphi\MC}\\&=-\frac{1}{2}\MC d^T\varphi+\frac{1}{2}\varphi\delta_{g^T}\SF\\&=-\frac{1}{2}\MC d^T\varphi+\frac{1}{2}\delta_{g^T}\p{\varphi \SF}+\frac{1}{2}\SF\cdot \nabla^T \varphi.\end{align*}
    One may show that linearising the divergence yields
    \begin{align*}
        \delta_{g^T}'\SF=-\delta_{g^T}\p{\varphi\SF}-\frac{n-1}{2}\SF\cdot \nabla^T \varphi+\frac{1}{2}\MC d^T\varphi.
    \end{align*}
    Combining the three preceding equations, we obtain 
    \begin{align*}
    0=\delta_{g^T}\SF'(h)-\frac{1}{2}\delta_{g^T}\p{\varphi\SF}-\frac{n-2}{2}\SF\cdot \nabla^T\varphi,
    \end{align*}
    which is precisely \eqref{hidden2} with $\lambda=0$. 
\end{remark}

\subsection{Stability}\label{sec:stab}
The decompositions of Proposition \ref{prop:split} are not fine enough for the purposes of stability. The main problem is that the kernel of $F_g$ contains the tangent space of the orbit of $g$ under the diffeomorphism group, $\Imag \delta^*_g$, cf. Lemma \ref{F_properties}. Though the Einstein equations are invariant under the entire diffeomorphism group, it turns out to be enough to discard the tangent space to the orbit of diffeomorphisms fixing the boundary. 

\begin{proposition}\label{bianchisplit} Let $(M,g)$ be an Einstein manifold with boundary, such that $\scal_g/(n-1)\notin \sigma_N(\Delta)$. Then the tangent space splits as
    \begin{equation}\label{stability_splitting}
        T_g\mathcal{M}_B=C_{N}^\infty(M)\cdot g\oplus \p{\ker_g P_g\cap \ker_g \beta_g}\oplus\Imag \delta^*_g|_{\Omega_0},
    \end{equation}
    according to which $F_g$ splits as 
    $$F_g=\p{-(n-2)P^*_g,\Delta_E,0}.$$
    The splitting \eqref{stability_splitting} is orthogonal with respect to $F_g$.
\end{proposition}
\begin{proof}
    The splitting of the tangent space follows directly from Proposition \ref{prop:split} and \cite[Lemma 2.2]{An1}: If $g$ is Einstein, then 
    $$T_g\mathcal{M}=\ker \beta\oplus \Imag \delta^*|_{\Omega_0}.$$
    The first and third factor in the splitting of $F$ is a consequence of Lemma \ref{F_properties}. The second factor is clear from the definition of $F$:
    $$Fh=\Delta_E h-2\delta^*\beta h-Ph\cdot g.$$
    Orthogonality follows from Lemma \ref{F_properties}.
\end{proof}

As in the closed case, the most interesting space is the one on which $F_g$ reduces to the Einstein operator $\Delta_E$. In our decomposition, this space is $\ker_gP_g\cap \ker_g\beta_g$. In analogy with the usual stability space, $TT_g$, we set
$$TV_g:=\ker_gP_g\cap \ker_g\beta_g.$$
Note that $TT_g$ is a subspace of $TV_g$.\\

Before proceeding, we shall take the time to show this constellation of operator, space and boundary conditions is suitable for spectral analysis.  

\begin{proposition}\label{prop:fredholm} Let $(M,g)$ be a compact Einstein manifold with boundary, and set 
$$Bh:=(B_1h,B_2h)=\p{h^T-\varphi g^T,2\MC'(h)+\varphi \MC},$$
viewed as a map from sections of $S^2M$ to sections of $S^2\partial M\times \partial M$. Define 
$$K^m:=\tb{h\in H^m\p{S^2M}\; :\;Ph=\beta h=0  }$$
and $K^m_B:=K^m\cap \ker B$. Then the operator 
    $$\Delta_E:K^{m+2}_B\to K^m$$ 
    is Fredholm and has discrete $L^2$-spectrum.
\end{proposition}
\begin{proof}
   To prove the proposition, one must first show that $\p{\Delta_E,\beta,B}$ defines a regular elliptic boundary value problem. This property is sometimes called the Shapiro--Lopatinsky condition, or the complementarity condition. For this triplet, the proof is essentially the same as the proof of \cite[Proposition 5.1]{RJJ}. For a more physical point of view of this process, we refer the reader to Section 2.2 of \cite{LSW}.  \\

   Given a normal coordinate neighbourhood $U$ of $p\in \partial M$, with $p$ mapped to the origin in $\R^n_+=\tb{x_n\geq 0}$. The method of "freezing coefficients" at $p$, means we consider the ordinary differential operators arising from the principal symbols, obtained by Fourier transforming only the tangential derivatives:
   \begin{align}
       \label{sigma_Delta_E}\sigma\p{\Delta_E}\p{h}_{ij}&=-\p{\frac{d^2}{d y^2}-\abs{\xi}^2}h_{ij},\\
       \label{symbolbetaalpha}\sigma(\beta)\p{h}_\alpha&=-\p{\frac{d}{dy}h_{n\alpha}+i\xi^\gamma h_{\gamma\alpha}-\frac{1}2i \xi_\alpha h^{j}_{\, j}},\\
        \label{symbolbetan}\sigma(\beta)\p{h}_n&=-\frac{1}{2}\p{\frac{d}{dy}\p{h_{nn}-h^{\alpha}_{\; \alpha}}+2i\xi^\alpha h_{\alpha n}},\\
       \label{symbolB1}\sigma(B_1)\p{h}_{\alpha\gamma}&=h_{\alpha\gamma }-(n-1)^{-1}h^{\delta}_{\; \delta},\\
       \label{symbolB2}\sigma(B_2)(h)&=\frac{d}{d y}h^{\alpha}_{\; \alpha} -2i\xi^\alpha h_{\alpha n},
   \end{align}
   where $\frac{d}{dy}=\frac{\partial}{\partial x_n}$, $\xi\in \R^{n-1}\setminus\tb{0}$, and Greek indices run from $1$ to $n-1$, while Latin indices run the gamut from $1$ to $n$. By showing that no non-trivial solution to $\sigma(\Delta_E)h(y)=\sigma(\beta)h(y)=0$ and $\sigma(B_1)h(0)=\sigma(B_2)h(0)=0$ can exists, we will have proven the the boundary value problem is well-posed as an elliptic problem. The bulk ODE \eqref{sigma_Delta_E} has bounded solutions 
   $$h_{ij}=a_{ij}(\xi)e^{-\abs{\xi}y}.$$
    Enforcing the Bianchi gauge condition \eqref{symbolbetaalpha}-\eqref{symbolbetan} on solutions of this form, forces the coefficient functions to satisfy
    \begin{align}
        \label{bianchi1}\abs{\xi}a_{\alpha n}-i\xi^\gamma a_{\gamma \alpha}+\frac{1}{2}i\xi_{\alpha}a^j_{\;j}&=0,\\
        \label{bianchi2}\frac{1}{2}\abs{\xi}\p{a_{nn}-a^\alpha_{\;\alpha}}-i\xi^\alpha h_{\alpha n}&=0.
    \end{align}
    We now impose the boundary conditions at $y=x_n=0$
    \begin{align}
        \label{sigmaB1}a_{\alpha\gamma}-(n-1)^{-1}a^{\delta}_{\; \delta}&=0,\\ \label{sigmaB2}\abs{\xi}a^{\alpha}_{\; \alpha}+2i\xi^\alpha a_{\alpha n}&=0.
    \end{align}
    Combining \eqref{bianchi2} and \eqref{sigmaB2}, we immediately see that $a_{nn}=0$. This means we can remove this term from the sum in \eqref{bianchi1}, and combine the equation with \eqref{sigmaB1} to obtain
    $$\abs{\xi}a_{\alpha n}+\frac{n-3}{2(n-1)}i\xi_{\alpha}a^{\delta}_{\; \delta}=0.$$
    By multiplying this with $\xi_\alpha/\abs{\xi}$ and summing over $\alpha$, we are left with 
    $$\xi^\alpha a_{\alpha n}+\frac{n-3}{2(n-1)}i\abs{\xi}a^{\delta}_{\; \delta}=0.$$
    This can easily be combined with \eqref{sigmaB1} and \eqref{sigmaB2} to show that 
    $$0=\xi^\alpha a_{\alpha n}+\frac{n-3}{2(n-1)}i\abs{\xi}a^{\delta}_{\; \delta}=\frac{n-2}{n-1}i\abs{\xi}a^{\delta}_{\; \delta}=(n-2)i\abs{\xi}a_{\alpha\gamma}.$$
    At this point only the mixed components $a_{\alpha n}$ can be non-trivial, but one can use \eqref{bianchi1} or \eqref{sigmaB2} (and other besides) to conclude that even these components vanish. Thus the triplet $\p{\Delta_E,\beta,B}$ satisfies the complementarity condition. \\

    From ellipticity of the pair $(\Delta_E,\beta)$ in the interior, one obtains an Agmon-Douglis-Nirenberg elliptic regularity estimate \cite[Theorem 10.5]{ADN}:
    $$\norm{h}_{H^{m+2}(M)}\leq C\p{\norm{\Delta_E h}_{H^m(M)}+\norm{\beta h}_{H^{m+1}(M)}+\norm{B_1h}_{H^{m+2-\frac{1}{2}}\p{\PS}}+\norm{B_2h}_{H^{m+1-\frac{1}{2}}(\PS)}+\norm{h}_{L^2(M)}}.$$
    On the space $K^{m+2}_B$, the estimate reduces to 
    \begin{equation}\label{ADNineq}\norm{h}_{H^{m+2}(M)}\leq C\p{\norm{\Delta_E h}_{H^m(M)}+\norm{h}_{L^2(M)}}.\end{equation}
    From this estimate, a standard argument shows that $\Delta_E$ has finite dimensional kernel and closed range. The argument goes as follows: Suppose $\ker\Delta_E\subset K_B^{m+2}$ is infinite dimensional, then we may choose a sequence $\tb{h_j}_{j}\subset K_B^{m+2}$ such that $\norm{h_j}_{H^{m+2}}=1$ for all $j$, but $h_j\rightharpoonup 0$ in $H^{m+2}$. By the Sobolev trace theorem, 
    $$B:H^{m+2}(S^2M)\longrightarrow H^{m+\frac{3}{2}}(S^2\partial M)\times  H^{m+\frac{1}2}(\PS)$$ is continuous for every $m\geq 0$. Hence $K_{B}^{m+2}$ is closed in $H^{m+2}$, and by Rellich--Kondrachov it embeds compactly into $ L^2$. Consequently, after passing to a subsequence, we can assume $h_j\rightarrow 0$ in $L^2$. However, this contradicts the ADN estimate \eqref{ADNineq}, as
    $$1=\norm{h_j}_{H^{m+2}}\leq C\norm{h_j}_{L^2}\longrightarrow 0.$$
    
    After concluding $\dim \ker \Delta_E<\infty$, we consider its $L^2$-orthogonal complement: $\p{\ker \Delta_E}^{\perp_{L^2}}\subset K^{m+2}_B$. We claim that on this subspace, the ADN estimate can be refined to 
    \begin{equation}\label{ADNrefinement}
        \norm{h}_{H^{m+2}}\leq C'\norm{\Delta_E h}_{H^m},\qquad \qquad \forall h\in \p{\ker \Delta_E}^{\perp_{L^2}},
    \end{equation}
    for some $C'>0$. Again we assume the converse, that no such $C'$ exists. For every $j\in \mathbb{N}$, we can find a $h_j\in \p{\ker \Delta_E}^{\perp_{L^2}}$ such that 
    $$\norm{h_j}_{H^{m+2}}>j\norm{\Delta_E h_j}_{H^m}, $$
    or, if we normalise $h_j$:
    $$\norm{\Delta_E h_j}_{H^m}< \frac{1}{j}.$$ 
    By possibly passing to a subsequence, we can assume that $h_j\rightharpoonup h\in K_B^{m+2}$. As the inclusion $K_B^{m+2}\hookrightarrow L^2$ is continuous, we even have $h_j\rightharpoonup h$ with respect to the $L^2$-inner product. Specifically, for every $k\in \ker\Delta_E\subset K_B^{m+2}$, 
    $$0=\lim_{j\to \infty}\p{h_j,k}_{L^2}=\p{h,k}_{L^2}.$$
    Meanwhile, the continuity of $\Delta_E$ implies $\Delta_E h=\lim_{j\to \infty}\Delta_E h_j=0$. Combining the preceding two statements, we find that $\norm{h}_{L^2}^2=0$. At this moment, we have a weakly convergent sequence $h_j\rightharpoonup 0$ in $K_B^{m+2}$, and may pass to a strongly convergent subsequence $h_j\rightarrow 0$ in $L^2$. This now contradicts the ADN estimate \eqref{ADNineq}:
    $$1=\norm{h_j}_{H^{m+2}}\leq C\p{\norm{\Delta_E h_j}_{H^m}+\norm{h_j}_{L^2}}\longrightarrow 0.$$
    This proves that \eqref{ADNrefinement} holds. Suppose now that we have a sequence $\big\{h'_j\big\}_j\subset K_B^{m+2}$ such that $\Delta_Eh'_j\to k\in H^m$. We may assume that $h'_j\in \p{\ker \Delta_E}^{\perp_{L^2}}$, so that \eqref{ADNrefinement} is valid for all $j$. Then 
    $$\norm{h'_j-h'_i}_{H^{m+2}}\leq C'\norm{\Delta_Eh'_j-\Delta_Eh'_i}_{H^m},$$
    which means $h'_j$ is Cauchy in $H^{m+2}$. Since $K_B^{m+2}$ is closed, it contains the limit $\lim_{j\to \infty} h'_j=h'$, for which $\Delta_E h'=k$. Thereby we have proved that $\Delta_E$ has closed range. As $\Delta_E$ coincides with $F_g$ on $K^{m+2}_B$, the Green identity \eqref{greens} shows that $\Delta_E$ is self-adjoint on this space. As a self-adjoint operator with finite dimensional kernel, $\dim\mathrm{coker}(\Delta_E)=\dim\ker(\Delta_E)<\infty$. Thus $\Delta_E:K^{m+2}_B\to K^m$ is Fredholm of index $0$. The fact that $\Delta_E$ maps $K^{m+2}_B$ to $K^m$ is a result of the fifth identity of Lemma \ref{F_properties}, and the commutation relation $\beta\circ \Delta_E=\p{\Delta_H-2\mu}\circ \beta $. \\
    
    As we noted earlier, $K^2_B$ embeds compactly in $K^0\subset L^2$, and as a consequence $\Delta_E$ has compact resolvent 
    $$\p{\Delta_E-\lambda}^{-1}:K^0\to K^0, \qquad \qquad \lambda\in \rho\p{\Delta_E}. $$ 
    Now the spectral theorem for compact operators implies that $\p{\Delta_E-\lambda}^{-1}$, and thus $\Delta_E$, has discrete spectrum with finite-dimensional eigenspaces.\end{proof}

We are finally in a position to define a notion of stability for Einstein metrics with boundary. As we need a way to distinguish between conformal stability and stability on $TV_g=\ker_gP\cap\ker_g\beta$, we shall have to find a descriptor for the later. Since it is the type of stability that most closely resembles the usual concept of mode stability, this is what we choose to call it.

\begin{definition} Let $(M,g)$ be a compact Einstein manifold with boundary. We say that: 
\begin{enumerate}
    \item $g$ is \textbf{conformally stable} (resp. \textbf{strictly conformally stable}) if $F_g$ is negative semi-definite (resp. definite) on $C_{g,N}^\infty(M)\cdot g$;
    $$\p{F_g(\psi g),\psi g}_{L^2}\leq 0\qquad \qquad \p{\text{resp. }<0}\qquad \qquad \forall \psi\in C_{g,N}^\infty(M)\setminus \tb{0}.$$
    \item $g$ is \textbf{mode stable} (resp. \textbf{strictly mode stable}) if $F_g$ is positive semi-definite (resp. definite) on $\ker_g P_g\cap \ker_g \beta_g$;
    $$\p{F_gh,h}_{L^2}\geq 0\qquad \p{\text{resp. }>0}\qquad \qquad \forall h\in \p{\ker_g P_g\cap\ker_g \beta_g}\setminus \tb{0}.$$
    \item $g$ is (strictly) \textbf{stable} if it is \textbf{both} (strictly) conformally stable and (strictly) mode stable.
    \item $g$ is (conformally/mode-) \textbf{unstable} if it is not (conformally/mode-) stable.
\end{enumerate}
\end{definition}
    Note that the definition is defined independently of the splitting \eqref{stability_splitting}. Indeed, in the case where $\scal_g/(n-1)\in \sigma_N(\Delta)$, we may still consider the metric stable/unstable, though it may clearly not be strictly stable.

    \begin{remark} As we saw in Remark \ref{infEinsteindef}, the kernel of $F_g$ encodes the same information, up to rescalings and isometries fixing the boundary, as the space of infinitesimal Einstein deformations, $\varepsilon(g)$. It should therefore come as no surprise, that strict mode stability implies infinitesimal rigidity in the usual sense: $\varepsilon(g)=\tb{0}$. We may see this in the following way: If $h\in \varepsilon(g)\subset \ker_gE'_g $ then $\Delta_Eh-2\delta^*\beta h=0$, which means 
    $$0=\tr\p{\Delta_Eh-2\delta^*\beta h}=2Ph.$$
    Thus $h\in \ker_g P$, and we may write $h=k+\delta^*\omega$ according to Proposition \ref{bianchisplit}, with $k\in \ker_g P\cap\ker_g\beta$ and $\omega\in \Omega_0$. Clearly, 
    $$Fk=Fh=\Delta_Eh-2\delta^*\beta h=0.$$
    So if $g$ is strictly mode stable, we must have $k=0$. This would imply that $h$ is pure gauge, $h\in \Imag \delta^*|_{\Omega_0}$, which is $L^2$-orthogonal to $\varepsilon(g)\subset \ker_g\delta$. The only possibility left is $h\equiv 0$, and since it was arbitrarily chosen, $\varepsilon(g)=\tb{0}$.
    \end{remark}

\begin{remark}
    It is clear that non-positive Einstein metrics are conformally stable. Furthermore, in the fortunate case where $\partial M$ is convex, Escobar's eigenvalue estimate \eqref{Obata} shows that 
    $$\p{F(\psi g),\psi g}_{L^2}=-(n-2)\p{\tr P^*\psi,\psi}_{L^2}\leq 0,$$
    and equality can occur if and only if $(M,g)$ is isometric to a hemisphere. Thus, conformal instability can only occur in positive Einstein metrics with non-convex boundary. See Examples \ref{example:spherical_strip} and \ref{example:spherical_cap} for two such cases. 
\end{remark}

Even though conformal instabilities may arise for positive Einstein metrics with non-convex boundary, we may still obtain an Obata-type uniqueness theorem that encompasses such metrics. At this point, we shall remind the reader that convexity in $\mathcal{M}_B$ can be considered a property of the conformal class.

\begin{theorem}\label{ObataTypeUniqueness} Let $g_1,g_2\in \mathcal{M}_B$ be Einstein metrics in the same conformal class, $\sq{g_1}=\sq{g_2}$. If $g_1$ is not isometric to $\p{\mathbb{S}^n_+(r),g_{\mathbb{S}^{n}}}$ for any $r>0$, then $g_1=cg_2$ for some $c>0$. If $g_1$ is isometric to $\p{\mathbb{S}^n_+(r),g_{\mathbb{S}^{n}}}$, then the same conclusion holds, up to isometry.
\end{theorem}
\begin{proof}
Suppose $g_2=\phi^{-2}g_1$, for some $0<\phi\in C_N^\infty(M)$, and consider the conformal transformation formula for the Einstein tensor:
$$E_{g_2}=E_{g_1}+(n-2)\phi^{-1}\p{\nabla^2_{g_1}\phi+\frac{\Delta_{g_1}\phi}{n}g_1}.$$
See e.g \cite[Theorem 1.159e]{Bes}. Since both metrics are assumed to be Einstein, their respective Einstein tensors vanish, leaving only 
$$\nabla^2_{g_1}\phi+\frac{\Delta_{g_1}\phi}{n}g_1=0.$$
If $\phi$ is constant, we are done. Assume, for contradiction, that $\psi:=\frac{\Delta_{g_1}\phi}{n}$ does not vanish identically. Using a Ricci formula,
$$d\psi=\delta_{g_1}\p{-\psi g_{1}}=\delta_{g_1}\nabla^2_{g_1}\phi=\Delta_{g_1}d\phi=d\Delta_{g_1}\phi-\mu_1d\phi=nd\psi-\mu_1d\phi.$$
Or, rearranged slightly,
\begin{equation}\label{eq:dpsi}d\psi=\frac{\mu_1}{n-1}d\phi.\end{equation}

If $g_1$ is Ricci-flat, $\mu_1=0$, we must have $\psi\equiv c$. By definition of $\psi$,
$$c=\frac{1}{\abs{M}_{g_1}}\int_M\psi\; dV_{g_1}=-\frac{1}{n\abs{M}_{g_1}}\int_{\PS}\nu_{g_1}(\phi)\; dA_{g_1}=0.$$
The only Neumann functions with $\Delta_{g_1} \phi=0$ are the constant ones.\\

If $\mu_1\neq 0$, it is apparent from \eqref{eq:dpsi} that $\psi\in C_{g,N}^\infty(M)$, and by taking the covariant derivative of both sides,
$$\nabla^2_{g_1}\psi=-\frac{\mu_1}{n-1}\psi g_1.$$
It may now be verified that $P^*_{g_1}\psi=0$, which by Theorem \ref{Escobar4.2} implies that $g_1$ is isometric to a hemisphere. In that case, Proposition \ref{properties_of_MB} implies that $\scal_{g_2}>0$, and the same argument as above shows that $(M,g_2)$ is likewise isometric to a hemisphere. Thus, up to the standard isometries of the hemisphere, $g_2$ is a rescaling of $g_1$.
\end{proof}

\section{Examples of Conformally Unstable Manifolds}\label{Sec:confunstab}

As mentioned, positive Einstein metrics with non-convex boundary need not be conformally stable. We illustrate this with two examples.

\begin{example}\label{example:spherical_strip} Consider the n-dimensional spherical strip, the middle section of the unit sphere, defined for $\varepsilon\in (0,1)$ by
$$\mathbb{S}^n_\varepsilon=\tb{x\in \mathbb{S}^n\subset \R^{n+1}\; :\;\abs{x_{n+1}}<\varepsilon}$$
We shall parametrise the standard round metric as the warped product:
$$g=d\theta^2+\cos^2\theta \;g_{\mathbb{S}^{n-1}},\qquad \qquad \theta\in [-\sin^{-1}\varepsilon,\sin^{-1}\varepsilon].$$
By considering a principal eigenfunction for the equatorial Laplacian, we may show that we can make the principal Neumann eigenvalue of $\mathbb{S}^{n}_\varepsilon$ as close that of $\mathbb{S}^{n-1}$ as we want, by choosing $\varepsilon$ small enough. Consider a standard spherical harmonic function on $\tb{x_{n+1}=0}\simeq \mathbb{S}^{n-1}$, satisfying
$$\Delta_{\mathbb{S}^{n-1}}Y=(n-1)Y,\qquad \norm{Y}_{L^2\p{\mathbb{S}^{n-1}}}=1,\qquad \int_{\mathbb{S}^{n-1}}Y\; dA=0.$$
Extend it trivially to a function on $\mathbb{S}^n_\varepsilon$,
$$f(\theta,\varphi):=Y(\varphi),\qquad \qquad \varphi=(\varphi_1,\ldots,\varphi_{n-1})\in \mathbb{S}^{n-1}.$$
This is a manifestly Neumann function with vanishing total value. Or, with respect to the usual notation, $f\in C^\infty_{g,N}\p{\mathbb{S}^n_\varepsilon}$. The principal Neumann eigenvalue of the Laplacian is given by the Rayleigh quotient 
$$\lambda_1=\inf_{\psi\in C_{g,N}^\infty\setminus \tb{0}}\mathcal{R}[\psi],\qquad \text{where}\qquad \mathcal{R}[\psi]:=\frac{\p{\psi,\Delta\psi}_{L^2}^2}{\norm{\psi}_{L^2}^2}=\frac{\norm{d\psi}_{L^2}^2}{\norm{\psi}_{L^2}^2}.$$
We may therefore obtain an upper bound for $\lambda_1$ by estimating $\mathcal{R}[f]$. The Laplacian of $g$ may be separated as
$$\Delta_g\psi=-\partial^2_{\theta\theta}\psi+(n-1)\tan\theta\; \partial_\theta \psi+\frac{1}{\cos^2\theta}\Delta_{\mathbb{S}^{n-1}}\psi.$$
As $\partial_\theta f=0$ and $Y$ is normalised, we have
\begin{equation}\label{eq:example1bound}\begin{split}\p{f,\Delta_g f}_{L^2}&=(n-1)\int_{-\sin^{-1}\varepsilon}^{\sin^{-1}\varepsilon}\cos^{n-3}\theta\; d\theta\\ &\leq \frac{n-1}{\cos^2\p{\sin^{-1}\varepsilon}}\int_{-\sin^{-1}\varepsilon}^{\sin^{-1}\varepsilon}\cos^{n-1}\theta\; d\theta=\frac{n-1}{1-\varepsilon^2}\norm{f}_{L^2}^2.\end{split}\end{equation}
This shows that, for $\varepsilon<\frac{1}{\sqrt{n}}$, we have 
$$\lambda_1\p{\mathbb{S}^n_\varepsilon}\leq \frac{n-1}{1-\varepsilon^2}<n=\lambda_1\p{\mathbb{S}^n},$$
and indeed 
$$\p{F(fg),fg}_{L^2}=-(n-2)(n-1)\p{(\Delta-n)f,f}_{L^2}>0.$$
As a final remark, we note that for $n=2$ the first integral of \eqref{eq:example1bound} may be evaluated directly to produce the sharper bound 
$$\lambda_1\pbig{\mathbb{S}^2_\varepsilon}\leq \frac{\ln\p{\frac{1+\varepsilon}{1-\varepsilon}}}{2\varepsilon}.$$
So, far from needing $\varepsilon <\frac{1}{\sqrt{2}}\sim 0.71$, we achieve the desired effect by removing very small caps. We simply require $\varepsilon<0.9575\ldots$ (the unique positive solution of $\ln\p{\frac{1+\varepsilon}{1-\varepsilon}}=4\varepsilon$). The sharper bound may then be carried over to an $n$-dimensional Einstein manifold, by taking the direct (non-warped) product $\mathbb{S}^2_\varepsilon\times \mathbb{S}^{n-2}$ with the usual product metric.
\end{example}

A reader familiar with the literature on this topic, might venture the guess that the conformal instabilities arise from the boundary being disconnected. 
The following example is modelled over the hemisphere and shows that the round metric immediately becomes conformally unstable if we extend the domain past the equator, where the boundary is no longer convex.

\begin{example}\label{example:spherical_cap}
    Consider now the (lower) spherical cap $\mathbb{S}^n_L$
    $$g=d\theta^2+\cos^2\theta \; g_{\mathbb{S}^{n-1}}\qquad \theta \in \sq{-\frac{\pi}{2},\frac{\pi}{2}L}$$
    for $L\in [0,1)$. With $Y$ as in the preceding example, we define the function
    $$f(\theta,\varphi)=\sin\p{\frac{\theta+\frac{\pi}{2}}{L+1}}Y(\varphi).$$
    It is not straightforward to estimate the Rayleigh quotient of $f$ directly, so we will instead show that it is decreasing at $L=0$, which will then imply that $\mathbb{S}^n_L$ is conformally unstable for $L>0$ small enough (the case $L=0$ is of course the hemisphere, which is stable, though not strictly stable). We consider the Rayleigh quotient of $f$ as a function of $L$,
    \begin{align*}R(L):=\frac{\p{L+1}^{-2}\int_{0}^\frac{\pi}{2}\cos^2\theta\sin^{n-1}(L+1)\theta\; d\theta+(n-1)\int_{0}^\frac{\pi}{2}\sin^2\theta\sin^{n-3}{(L+1)\theta}\; d\theta}{\int_{0}^\frac{\pi}{2}\sin^2\theta \sin^{n-1}{(L+1)\theta}\; d\theta}.\end{align*}
    Using 
    $$\int_{0}^\frac{\pi}{2}\sin^{n+1}\theta\; d\theta=\frac{n}{n+1}\int_{0}^\frac{\pi}{2}\sin^{n-1}\theta\; d\theta,$$
    one may easily verify that 
    $$R(0)=\frac{\p{n-\frac{n}{n+1}}\int_{0}^\frac{\pi}{2}\sin^{n-1}\theta\; d\theta}{\frac{n}{n+1}\int_{0}^\frac{\pi}{2}\sin^{n-1}\theta\; d\theta}=n=\lambda_1\p{\mathbb{S}^n_+}.$$
    In particular, $R$ coincides with the principal Neumann eigenvalue at $L=0$. One may then apply 
    $$\int_{0}^\frac{\pi}{2}\theta\cos\theta\sin^{n}\theta\; d\theta=\frac{\pi}{2(n+1)}-\frac{1}{n+1}\int_{0}^\frac{\pi}{2}\sin^{n+1}\theta\; d\theta,$$
    to show that
    $$R'(0)=\frac{-\frac{\pi}{2}\frac{n}{n+1}\int_{0}^\frac{\pi}{2}\sin^{n-1}\theta\; d\theta}{\p{\frac{n}{n+1}\int_{0}^\frac{\pi}{2}\sin^{n-1}\theta\; d\theta}^2}=-\frac{\pi(n+1)}{2n}\p{\int_{0}^\frac{\pi}{2}\sin^{n-1}\theta\; d\theta}^{-1}<0.$$
    This implies that $\lambda_1(\mathbb{S}^n_L)<n$ for $L>0$ small. As in the preceding example, this means that $\p{\mathbb{S}^n_L,g_{\mathbb{S}^n}}$ is conformally unstable.\\

    Though the integrals are difficult to evaluate algebraically, we include a plot below of the numeric values of the upper bound $R(L)-n$ for a range of dimensions.
    \begin{figure}[ht]
    \centering
    \def\svgwidth{0.8\textwidth}
    %% Creator: Inkscape 1.4 (e7c3feb1, 2024-10-09), www.inkscape.org
%% Accompanies image file 'sphericalcap.pdf' (pdf, eps, ps)
\begingroup%
  \makeatletter%
  \providecommand\color[2][]{%
    \errmessage{(Inkscape) Color is used for the text in Inkscape, but the package 'color.sty' is not loaded}%
    \renewcommand\color[2][]{}%
  }%
  \providecommand\transparent[1]{%
    \errmessage{(Inkscape) Transparency is used (non-zero) for the text in Inkscape, but the package 'transparent.sty' is not loaded}%
    \renewcommand\transparent[1]{}%
  }%
  \providecommand\rotatebox[2]{#2}%
  \newcommand*\fsize{\dimexpr\f@size pt\relax}%
  \newcommand*\lineheight[1]{\fontsize{\fsize}{#1\fsize}\selectfont}%
  \ifx\svgwidth\undefined%
    \setlength{\unitlength}{578.28410952bp}%
    \ifx\svgscale\undefined%
      \relax%
    \else%
      \setlength{\unitlength}{\unitlength * \real{\svgscale}}%
    \fi%
  \else%
    \setlength{\unitlength}{\svgwidth}%
  \fi%
  \global\let\svgwidth\undefined%
  \global\let\svgscale\undefined%
  \makeatother%
  \begin{picture}(1,0.32870109)%
    \lineheight{1}%
    \setlength\tabcolsep{0pt}%
    \put(0,0){\includegraphics[width=\unitlength,page=1]{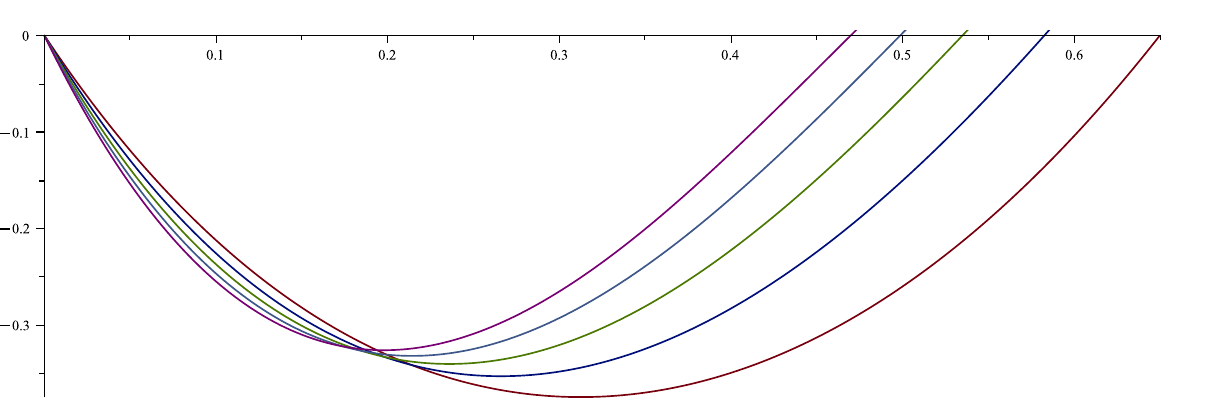}}%
    \put(0.45593195,0.25376155){\color[rgb]{0,0,0}\makebox(0,0)[lt]{\lineheight{1.25}\smash{\begin{tabular}[t]{l}$L$\end{tabular}}}}%
    \put(0.65076215,0.31116582){\color[rgb]{0.47058824,0,0.44705882}\makebox(0,0)[lt]{\lineheight{1.25}\smash{\begin{tabular}[t]{l}$n=7$\end{tabular}}}}%
    \put(0.74336852,0.31116582){\color[rgb]{0.24313725,0.34117647,0.54117647}\makebox(0,0)[lt]{\lineheight{1.25}\smash{\begin{tabular}[t]{l}$6$\end{tabular}}}}%
    \put(0.79427509,0.31116581){\color[rgb]{0.29019608,0.47058824,0}\makebox(0,0)[lt]{\lineheight{1.25}\smash{\begin{tabular}[t]{l}$5$\end{tabular}}}}%
    \put(0.86148824,0.31116581){\color[rgb]{0,0.05490196,0.47058824}\makebox(0,0)[lt]{\lineheight{1.25}\smash{\begin{tabular}[t]{l}$4$\end{tabular}}}}%
    \put(0.95363761,0.31116581){\color[rgb]{0.47058824,0,0.05490196}\makebox(0,0)[lt]{\lineheight{1.25}\smash{\begin{tabular}[t]{l}$3$\end{tabular}}}}%
  \end{picture}%
\endgroup%

   \caption{$R(L)-n$ for $n=3,\ldots,7$ and $L\in [0,0.65)$. This shows that the unstable range is actually quite large; every manifold between the hemisphere and the lower "three-quarter" sphere (at least for $3\leq n\leq 5$).}
\end{figure}
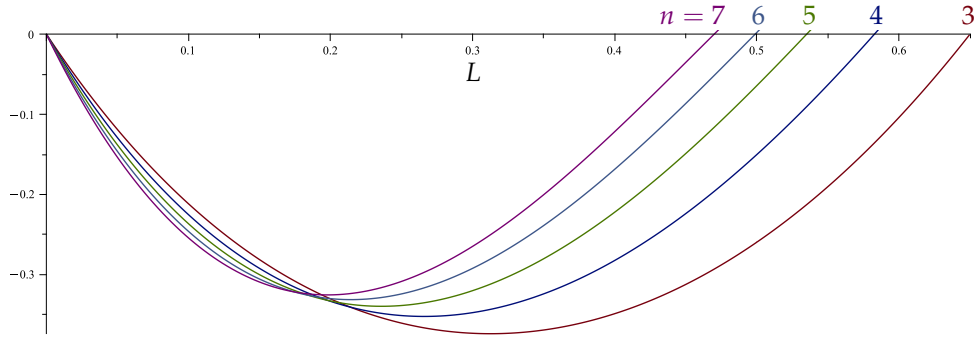
\end{example}

\break

\section{Schwarzschild-AdS}\label{sec:SADS}

In this section, we consider the family of $n$ dimensional Schwarzschild anti-deSitter (SAdS) metrics
\begin{equation}\label{SADSmetric}g=V(r)^{-1}\; dr^2+V(r)\; dt^2+r^2\; g_{\mathbb{S}^{n-2}},\qquad \qquad V(r)=1-\frac{2m}{r^{n-3}}+\frac{\mu r^2}{n-1}.\end{equation}
These form a classic family of examples of non-positive Einstein metrics with $\Ric_g=-\mu g$ ($\mu\geq 0$), modelling an isolated black hole in a curved background. Taking $\mu = 0$, we get the usual Ricci-flat Schwarzschild metric. The metrics are á priori defined on the manifold 
$$(r_0,\infty)\times \R\times \mathbb{S}^2$$
where $r_0$ is the unique positive root of $V$, corresponding to the event horizon of black hole. In the Riemannian setting, we may extend the metric to the coordinate singularity $r=r_0$, by compactifying the temporal dimension.\\

As $r_0$ is an isolated singularity of multiplicity one, near the singularity we have
$$V(r)\sim V'(r_0)(r-r_0).$$
By a change of variable 
$$\rho=\int_{r_0}^r\frac{1}{V(s)^{1/2}}\;ds\sim \sqrt{\frac{4(r-r_0)}{V'(r_0)}} $$
one may express $g$ as 
$$g=d\rho^2+\rho^2 U(\rho)\; dt^2+W(\rho)\; g_{\mathbb{S}^{n-2}},$$
where $W(\rho)\to r_{0}^2$ as $\rho\to 0$ and
$$U(\rho)=\frac{V(r)}{\rho^2}\sim \frac{V'(r_0)(r-r_0)}{\rho^2}\sim\frac{V'(r_0)^2}{4},$$ 
near the singularity. As the spherical part behaves nicely, our goal is to avoid a conical singularity in the $(\rho,t)$ plane. This requires a periodic identification $t\sim t+\frac{4\pi}{V'(r_0)}$, such that if we set $t=\frac{2\tau}{V'(r_0)}$ we have 
$$g\sim d\rho^2+\rho^2\; d\tau^2+W(\rho)\; g_{\mathbb{S}^{n-2}},$$
with the manifold collapsing to an $(n-2)$-sphere for $\rho\to 0$.\\

From now on, we consider the family of metrics \eqref{SADSmetric} on the manifold
$$M_R:=[r_0,R]\times \R/\beta\mathbb{Z}\times \mathbb{S}^2,\qquad \qquad \beta=\frac{4\pi}{V'(r_0)}.$$
Note the addition of an upper limit on the range of $r$, the cavity radius $R>r_0$. This we allow to vary, so as to explore the effect of the boundary conditions imposed at different radii. \\

The period $\beta$ has an interesting physical interpretation as the reciprocal temperature measured at infinity. By the Tolman law, a local observer at distance $r>r_0$ will measure a local temperature of 
\begin{equation}\label{Tolmantemp}T(r)=\frac{1}{\beta\sqrt{V(r)}}.\end{equation}

An alternative notion of stability for black holes is that of \textit{thermodynamic stability}, see \cite{Pre}. In the setting of a Schwarzschild or Schwarzschild anti-deSitter black hole in a spherical cavity with radius $R>r_0$, one considers it thermodynamically stable if the heat capacity 
$$C_R=\p{\frac{\partial E}{\partial T}}_R$$
is positive. Here $T=T(R,m)$ is the Tolman temperature, \eqref{Tolmantemp}, and $E$ is the Brown-York quasilocal energy, \cite{BY},
$$E(R,m)=R\p{\sqrt{1+\frac{\mu R^2}{n-1}}-\sqrt{1-\frac{2m}{R}+\frac{\mu R^2}{n-1}}}.$$
To give us an idea of the cavity size $R$, for which an SAdS black hole of mass $m$ and cosmological constant $\mu$, goes from being thermodynamically stable to unstable, we shall estimate the value at which $C_R$ changes sign. We do this only for $n=4$. By definition, $r_0$ is a solution of
$$m=\frac{r_0}{2}\p{1+\frac{\mu r_0^2}{3}}.$$
We may therefore consider both $E$ and $T$ as functions in $R$ and $r_0$. Then one may calculate 
$$\frac{\partial E}{\partial r_0}=\frac{1+\mu r_0^2}{2\sqrt{V(R)}},$$
and 
$$\frac{\partial T}{\partial r_0}=\frac{2\mu\p{\mu r_0^2-1}R^3+6\p{\mu r_0^2-1}R+\mu^2r_0^5+2\mu r_0^3+9r_0}{8\pi \sqrt{V(R)}r_0^2(R-r_0)(\mu R^2 + \mu r_0 R + \mu r_0^2 + 3)}.$$
Notice that $\frac{\partial E}{\partial r_0}$ is positive for all values of $R>r_0$. The same goes for the denominator of $\frac{\partial T}{\partial r_0}$. We surmise that 
$$C_R=\frac{\partial E/\partial r_0}{\partial T/\partial r_0}>0\qquad \iff \qquad 2\mu\p{\mu r_0^2-1}R^3+6\p{\mu r_0^2-1}R+\mu^2r_0^5+2\mu r_0^3+9r_0>0.$$
The right expression has a single positive root $R_0>r_0$, where it goes from being positive to being negative. Substituting $m$ back, we can estimate this point as 
$$R_0=3m+\frac{56}{3}m^3\mu +\mathcal{O}(\mu^2).$$
This shows that a Schwarzschild black hole of mass $m$ is thermodynamically stable in a cavity of radius $r_0|_{\mu=0}=2m<R<R_0|_{\mu=0}=3m$, and unstable for $R>3m$. We also see that, at least for small values of $\mu>0$, the range of thermodynamically stable cavity sizes increases with $\mu$. As we shall see, this is exactly the behaviour we observe with respect to our notion of (mode) stability.\\

In their original construction of a negative mode for the linearised Einstein equations on the four dimensional Schwarzschild manifold \cite{GPY}, Gross, Perry and Yaffe proved that such a mode must lie in the lowest frequency, spherically symmetric sector of $TT_g\subset TV_g$. The same should be true in higher dimensions, and in the presence of a negative cosmological constant $-\mu<0$, as argued in \cite{Pre}. As it is beyond the spectrum of the present paper to prove that the same holds for $TV_g$, we shall be content with including it as a caveat in the statement of the first theorem of this section:

\begin{theorem}\label{thm:SADS}
    Every member of the SAdS family is (spherically symmetrically) stable on
    $$M_R=[r_0,R]\times \R/\beta\mathbb{Z}\times \mathbb{S}^{n-2},$$
    when $R=\p{(n-1)m}^{1/(n-3)}$. The stability is strict for $\mu \neq 0$.
\end{theorem}

To begin the search for negative eigenmodes for $\Delta_E$, we start with an arbitrary spherical symmetric $h\in TV_g$. The most general form of such a metric perturbation is
$$h=h_{rr}(r,t)\;dr^2+2h_{rt}(r,t)\; dr\;dt+h_{tt}(r,t)\;dt^2+k(r,t)\; r^2g_{\mathbb{S}^{n-2}}.$$
See e.g. \cite{KI}. These are precisely the modes that preserve the spatial symmetry of the metric.
Since $g$ is static, so that $\Delta_E$ is invariant under translations in $t$, $h$ admits a Fourier expansion 
$$h=\sum_{k\in \mathbb{Z}}h_k(r)\exp\p{\frac{2\pi i k t}{\beta}}.$$
As the Fourier modes decouple for $\Delta_E$, the smallest possible eigenvalue will have to lie in the static ($k=0$) sector. Specifically, we may assume that $h$ is purely radial, $h=h(r)$. If we let $g^T$ be the induced metric on $\partial M_R=\tb{r=R}\times \R/\beta\mathbb{Z}\times \mathbb{S}^{n-2}$, we get that 
 $$\delta_{g^T}(h\cdot \nu)=-V(R)^{-1/2}\partial_{t}h_{tr}=0.$$ 
By imposing the conformal Neumann condition \eqref{BC2nongeometric} and the Bianchi gauge on $h$, one finds that  
$$0=\g{\delta h+d\tr h}{\nu}+\delta_{g^T}\p{h\cdot \nu}=\frac{1}{2}\g{d\tr h}{\nu}.$$
This is an elliptic condition on the trace, for the Bianchi-gauged equation $Ph=0$. That is, $\tr h$ satisfies 
$$2Ph=\p{\Delta+2\mu}\tr h=0\; \; \text{in}\;\; M ,\qquad \qquad \nu(\tr h)=0\;\;\text{on}\;\; \PS,$$
which implies the pointwise $\tr h=0$. Now, a traceless tensor tensor $h\in TV_g$, must lie in the subspace $TT_g$.\\

By considering the $t$ component of the divergence of $h$, we may solve for the off-diagonal components of $h$ explicitly. Combined with the previous simplifications, we find that our candidate negative mode must be of the form
$$h=\chi(r)V(r)^{-1}\;dr^2+2Cr^{2-n}V(r)^{-1}\; dr\;dt+\psi(r)V(r)\;dt^2+k(r)\; r^2g_{\mathbb{S}^{n-2}},$$
where $\chi(r)$, $\psi(r)$ and $k(r)$ are unknown radial functions, and $C\in \R$. It does not make sense to impose boundary conditions at the event horizon $r=r_0$, but we do require that the metric perturbation preserves the smooth extension. To that end, we expand the tensor $g+\varepsilon h$ ($0<\varepsilon\ll1$) near $r=r_0$. This is done in the same manner as when we extended the metric itself. In this case, to first order
$$g+\varepsilon h\sim d\rho^2+\frac{1+\varepsilon\psi(r_0)}{1+\varepsilon \chi(r_0)}\rho^2\; d\tau^2+\frac{8 C\varepsilon}{r_0^{n-2}V'(r_0)^2(1+\varepsilon\chi(r_0))^{1/2}\rho}\; d\rho \;d\tau+\p{1+\varepsilon k(r_0)}r_{0}^2\; g_{\mathbb{S}^{n-2}}.$$
We see that the periodicity and regularity is preserved if and only if $\psi(r_0)=\chi(r_0)$ and $C=0$.\\

The preceding discussion shows that we can restrict to diagonal perturbations of the form
\begin{equation}\label{diagonalformofh} h=\chi(r)V(r)^{-1}\;dr^2+\psi(r)V(r)\;dt^2-\frac{\psi(r)+\chi(r)}{n-2}\; r^2g_{\mathbb{S}^{n-2}}.\end{equation}
Such a $h$ is inherently traceless, while being transverse requires 
$$0=\delta h=-\p{\chi'(r)+\frac{rV'(r)+2(n-1)V(r)}{2rV(r)}\chi(r)-\frac{rV'(r)-2V(r)}{2rV(r)}\psi(r)}dr.$$
That is, we may express $\psi$ via an ODE in $\chi$:
\begin{equation}\label{diveq}\psi(r)=\frac{2rV(r)}{rV'(r)-2V(r)}\chi'(r)+\frac{rV'(r)+2(n-1)V(r)}{rV'(r)-2V(r)}\chi(r).\end{equation}
Note that $rV'(r)-2V(r)$ vanishes exactly at the umbilic boundary $R=\p{(n-1)m}^{1/(n-3)}$. \\

A straightforward calculation gives 
$$\p{\nabla^*\nabla h}_{rr}=-\p{V\chi''+\frac{rV'+2V}{r}\chi'-\frac{r^2\p{V'}^2+4(n-1)V^2}{2r^2V}\chi+\frac{r^2\p{V'}^2-4V^2}{2r^2V}\psi}g_{rr}$$
while 
$$2(\overset{\circ}{R}h)_{rr}=\p{\frac{V'}{r}\chi-\frac{rV''-V'}{r}\psi}g_{rr}$$
These combine to
\begin{align*}\p{\Delta_Eh}_{rr}=-\Bigg(V\chi''+\frac{rV'+2V}{r}\chi'&-\frac{r^2\p{V'}^2+4(n-1)V^2-2rVV'}{2r^2V}\chi\\&\qquad +\frac{r^2\p{V'}^2-4V^2+2rV\p{V'-rV''}}{2r^2V}\psi\Bigg)g_{rr}.\end{align*}
If one then inserts the divergence constraint \eqref{diveq}, it becomes a second order ODE in $\chi$:
\begin{align*}\p{\Delta_Eh}_{rr}=-\Bigg(V\chi''&-\frac{8V^2+2rV\p{rV''-V'}-2r^2(V')^2}{r\p{rV'-2V}}\chi'\\ &-\frac{4VV'-r\p{(n+2)(V')^2-2(n-1)VV''}+r^2V'V''}{r\p{rV'-2V}}\chi\Bigg) g_{rr}.\end{align*}

As $\Delta_E$ maps $TT_g$ to itself, the tensor $k:=\Delta_Eh-\lambda h$ must necessarily also be of the form of \eqref{diagonalformofh} and satisfy \eqref{diveq}. Thus, if it is possible to find a $\lambda$ such that $k_{rr}=0$, the entirety of $k$ must vanish. We can therefore get away with only solving $k_{rr}=0$, the master equation.\\

For the remainder of the explicit calculations, we shall set $n=4$. This is foremost for the sake of readability, but also to allow us to exhibit an alternative approach in the case of $n\geq 5$. In four dimensions the master equation can be written as   
\begin{equation}\label{master}P(r)\chi''(r)+Q(r)\chi'(r)+R(r)\chi(r)=0,\end{equation}
where
\begin{align*}
    P(r):=&-\frac{r\p{rV'-2V}}2\\
    Q(r):=&-\frac{2r^2(V')^2-2rV\p{rV''-V'}-8V^2}{2V}\\
    R(r):=&\frac{r^2\p{V'V''-\lambda V'}-2r\p{3(V')^2-3VV''-\lambda V}+4VV'}{2V}.
\end{align*}
It is clear that the equation has two regular singular points in its domain, one at $r=r_0$ where $V(r)=0$, and one at $r =3m$, the single positive root of $rV'-2V$. We may develop Frobenius solutions around either of these, but $r=3m$ will be the most interesting case for us.\\ 

We shall, in fact, not have to concern ourselves overmuch with the event horizon $r=r_0$ at all. For any $h$ given by \eqref{diagonalformofh}, the divergence constraint, \eqref{diveq}, implies
$$\psi(r_0)=\frac{2r_0V(r_0)}{r_0V'(r_0)-2V(r_0)}\chi'(r_0)+\frac{r_0V'(r_0)+6V(r_0)}{r_0V'(r_0)-2V(r_0)}\chi(r_0)=\frac{r_0V'(r_0)}{r_0V'(r_0)}\chi(r_0)=\chi(r_0),$$
since $V(r_0)=0$. So, as long as $\chi$ and $\psi$ are themselves regular at $r_0$, the perturbation preserves the regularity of the metric.\\

By developing a Frobenius series around the other singular point, $r=3m$, we may consider the effect of various boundary conditions, imposed at - or near - this radius. The coefficient functions are developed as Taylor series around $r_c=3m$:
\begin{align*}
    P(r):=&r-3m\\
    Q(r):=&-2+\frac{10}{3m}(r-3m) -\frac{2\p{11 + 45 \mu m^2}}{9m^2\p{1 + 9\mu m^2}}(r-3m)^2+\cdots\\
    R(r):=&-\frac{8}{3m} +\frac{27\lambda m^2 + 32 + 234\mu m^2}{9m^2\p{1 + 9\mu m^2}}(r-3m)-\frac{2\p{27 \lambda m^2 + 52  + 198\mu m^2}}{27m^3\p{1 + 9\mu m^2}}(r-3m)^2+\cdots
\end{align*}
Assuming a solution to \eqref{master} of the type
$$\chi(r)=(r-3m)^k\sum_{\ell=0}^\infty a_\ell(r-3m)^{\ell},$$
we may obtain the possible values of $k$ by solving the indicial equation:
$$k^2+\p{\frac{Q(3m)}{P'(3m)}-1}k=k(k-3).$$
The greater root, $k=3$, will always provide a solution, while the smaller, $k=0$, will only give a regular solution if it is possible to solve for the third coefficient. Assuming $k=0$, we find the following recursive relations for the first few coefficients
\begin{align*}
    a_1=-\frac{4a_0}{3m},\qquad \qquad 
    a_2=\frac{\p{\p{27\lambda + 234\mu}m^2 + 32}a_0 + 6m\p{1 + 9\mu m^2}a_1}{18m^2\p{ 1+ 9\mu m^2}},\qquad \qquad a_3=\text{'arbitrary'}.
\end{align*}
With no restriction on $a_3$, the $k=0$ solution is regular as well. In fact, if we impose our boundary conditions at $R=3m$, the possible values of $\lambda$ depend solely on this solution. Setting $a_0=1$ and $a_3=0$, we obtain  
\begin{equation}\label{chisol}\chi(r)=1-\frac{4}{3m}(r-3m)+\frac{9\lambda m^2  + 54\mu m^2+ 8}{6m^2\p{1+ 9\mu m^2}}(r-3m)^2+\mathcal{O}\p{(r-3m)^4}.\end{equation}

Recall that it was the imposition of \eqref{BC2nongeometric} that proved that $h$ had to be traceless. We should therefore be able to pin down $\lambda$ by imposing \eqref{BC1}, conformal preservation of the induced metric on $\PS_R$. For this, we require
$$h^T=\psi(R)V(R)\; dt^2-\frac{\psi(R)+\chi(R)}{2}r^2g_{\mathbb{S}^2}=\varphi\p{V(R)\; dt^2+r^2g_{\mathbb{S}^2}}=\varphi g^T,$$
for some $\varphi\in \R$. In other words, we must have 
\begin{equation}\label{psiequalsthirdchi}\psi(R)=-\frac{\chi(R)}{3}.\end{equation}

Expanding \eqref{diveq} to first order around $3m$, we obtain
\begin{equation*}
    \psi(r)=\frac{18\mu m^2-27\lambda m^2-1}{3}-\frac{2\p{72\mu m^2+27\lambda m^2+16}}{3m}(r-3m)+\cdots 
\end{equation*}
Imposing \eqref{psiequalsthirdchi} at $r=R\sim 3m$ yields
\begin{equation}\label{bctaylor}0=\psi(R)+\frac{\chi(R)}3=3m^2\p{2\mu-3\lambda}-\frac{2\p{216\mu m^2+81\lambda m^2 + 50 }}{9m}(R-3m)+\cdots\end{equation}
At $R=3m$, the only possible solution is 
$$\lambda=\frac{2\mu}{3}.$$
This proves the principal eigenvalue of $\Delta_E$, restricted to spherically symmetric modes, is strictly positive for all values of $\mu>0$. Furthermore, the same expression gives the principal eigenvalue for the standard Schwarzschild metric: $\lambda_{SC}=0$. We shall also prove that $\lambda_{SC}$ is strictly decreasing as a function of $R$, which provides grounds for the following theorem:
\begin{theorem}\label{thm:S}
    The four dimensional Schwarzschild metric ($\mu =0$) becomes unstable when the cavity radius passes the photon sphere $R=3m$. Concretely, there exists $r_0<L\leq \infty$ such that $g_{SC}$ is unstable on $M_R$ for all $R\in (r_0,L)$.
\end{theorem}

This method does not allow us to prove directly, that even for $\mu>0$ does the metric develop an instability if the radius of the cavity is increased. However, numerical estimation suggest that this is at least the case for small values of $\mu$:

\begin{figure}[ht]
    \centering\hspace{0.5cm}
    \def\svgwidth{0.9\textwidth}
    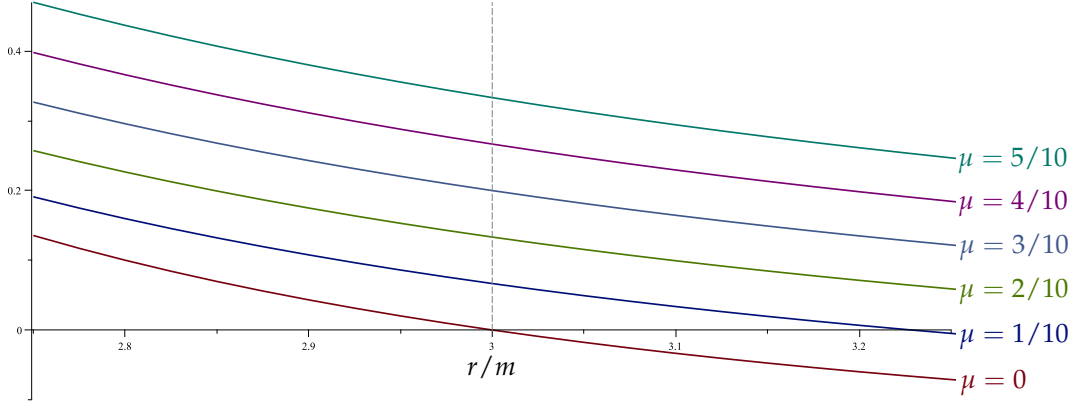
   \caption{Numerical estimates of the lowest eigenvalue in dimension $n=4$. The values are exact at $r=3m$ (dotted line), and the precision decreases as $(r-3m)^3$.}
\end{figure}

The Schwarzschild metric exhibits a similar behaviour to that which we saw for the conformal perturbations of the hemisphere in Example \ref{example:spherical_cap}. Specifically, the stability operator has non-trivial kernel when $R=3m$, is positive definite for slightly smaller $R$ and has a negative eigenvalue for slightly larger $R$. This can be proven directly from \eqref{bctaylor}, as it shows that for $R\sim 3m$,
$$\lambda_{SC}\sim -\frac{100(R-3m)}{(9m)^2\p{m+2(R-3m)}}.$$
In particular,
$$\left.\frac{d\lambda_{SC}}{dR}\right|_{R=3m}=-\frac{100}{(9m)^2m},$$
which shows that the metric develops its instability exactly at $R=3m$.\\

The calculations above can be replicated for $n>4$ and $R=((n-1)m)^{1/(n-3)}$, but we shall take a different approach. Note that the indicial equation is independent of the dimension. In particular, one may develop Frobenius solutions of order $k=0$ and $k=3$ at the regular singular point $R=\p{(n-1)m}^{1/(n-3)}$ for all $n\geq 4$. As before, the spectrum only depends on the $k=0$ solution, which implies that $h$ does not vanish identically on the boundary. We may then use a constraint condition to find the eigenvalue corresponding to a solution of this type. Since $\partial M_R$ is totally umbilic and $\tr h=0$, \eqref{hidden1} simplifies to 
$$\Delta_{g^T}\varphi+\mu\varphi-\frac{n-1}{n-2}\lambda\varphi=0,$$
where we also corrected the sign of the Einstein constant to be consistent with the definition of this section. Since we are only considering radial eigenmodes, we have $\Delta_{g^T}\varphi=0$. As $\varphi$ can not vanish identically along the boundary, it means 
$$\lambda=\frac{n-2}{n-1}\mu,$$
which is consistent with the explicit result we got for $n=4$.

\pagebreak

\end{document}